\documentclass[10pt,reqno]{amsart}
\usepackage{hyperref}
\usepackage{latexsym,amsmath}
\usepackage{enumerate}
\usepackage{amsfonts}
\usepackage{amssymb}
\usepackage{latexsym}
\usepackage{fixmath}
\usepackage[usenames,dvipsnames]{color}
\usepackage{mathdots}
\usepackage{nicematrix}

\usepackage{graphicx}
\usepackage{comment}

\usepackage[T1]{fontenc}
\usepackage{fourier}
\usepackage{bbm}
\usepackage{color}
\usepackage{fullpage}

\newcommand{\trans}{\top}

\numberwithin{equation}{section}

\newcommand{\set}[1]{\left\{#1\right\}}
\newcommand{\bDelta}{\bar\Delta}
\newcommand{\bGamma}{\bar\Gamma}
\newcommand{\vDelta}{\vec\Delta}
\newcommand{\vGamma}{\vec\Gamma}
\newcommand{\frozen}{\mathtt{f}}
\newcommand{\slush}{\mathtt{s}}
\newcommand{\unfrozen}{\mathtt{u}}

\renewcommand{\vec}[1]{\boldsymbol{#1}}

\renewcommand{\subset}{\subseteq}

\newcommand{\WP}{\mathrm{WP}}

\newcommand\vhy{\vec{y}^\dagger}
\newcommand\vhs{\vec\sigma^\dagger}
\newcommand\vhA{\vA^\dagger}
\newcommand\vha{\vec\alpha}

\newcommand\CPC{Combinatorics, Probability and Computing}

\newcommand\nix{\,\cdot\,}

\newcommand\cA{\mathcal A}
\newcommand\cB{\mathcal B}

\newcommand\cD{\mathcal D}
\newcommand\cE{\mathcal E}
\newcommand\cF{\mathcal F}
\newcommand\cG{\mathcal G}

\newcommand\cI{\mathcal I}

\newcommand\cL{\mathcal L}

\newcommand\cR{\mathcal R}

\newcommand\cX{\mathcal X}
\newcommand\cY{\mathcal Y}

\newcommand\fA{\mathfrak A}

\newcommand\fC{\mathfrak C}
\newcommand\fD{\mathfrak D}
\newcommand\fE{\mathfrak E}
\newcommand\fF{\mathfrak F}

\newcommand\fI{\mathfrak I}

\newcommand\fM{\mathfrak M}

\newcommand\fP{\mathfrak P}

\newcommand\fR{\mathfrak R}
\newcommand\fS{\mathfrak S}

\newcommand\fV{\mathfrak V}

\newcommand\fX{\mathfrak X}

\newcommand\fh{\mathfrak h}

\newcommand\fl{\mathfrak l}
\newcommand\fm{\mathfrak m}

\newcommand\fp{\mathfrak p}

\newcommand\vA{\vec A}

\newcommand\vF{\vec F}

\newcommand\vX{\vec X}
\newcommand\vY{\vec Y}
\newcommand\vZ{\vec Z}
\newcommand\va{\vec a}

\newcommand\ve{\vec e}

\newcommand\vj{\vec j}

\newcommand\vm{\vec m}

\newcommand\vt{\vec t}

\newcommand\vv{\vec v}

\newcommand\vy{\vec y}

\newcommand\eps{\varepsilon}

\newcommand\FF{\mathbb{F}}
\newcommand\field{\FF}
\newcommand\NN{\mathbb{N}}
\newcommand\ZZpos{\mathbb{Z}_{\geq0}}

\newcommand\Erw{\mathbb{E}}
\newcommand\ex{\Erw}
\newcommand{\vecone}{\mathbb{1}}

\newcommand{\Po}{{\rm Po}}
\newcommand{\Bin}{{\rm Bin}}

\newcommand{\Be}{{\rm Be}}

\newcommand\bc[1]{\left({#1}\right)}
\newcommand\cbc[1]{\left\{{#1}\right\}}

\newcommand\brk[1]{\left\lbrack{#1}\right\rbrack}

\newcommand\abs[1]{\left|{#1}\right|}

\newcommand{\Whp}{W.h.p.}
\newcommand{\whp}{w.h.p.}

\newcommand\pr{\mathbb{P}} 
 
\newcommand\Lem{Lemma}
\newcommand\Prop{Proposition}
\newcommand\Thm{Theorem}

\newcommand\Cor{Corollary}
\newcommand\Sec{Section}

\newtheorem{definition}{Definition}[section]
\newtheorem{claim}[definition]{Claim}

\newtheorem{theorem}[definition]{Theorem}
\newtheorem{lemma}[definition]{Lemma}
\newtheorem{proposition}[definition]{Proposition}
\newtheorem{corollary}[definition]{Corollary}

\newtheorem{fact}[definition]{Fact}

\DeclareMathOperator{\nul}{nul}
\DeclareMathOperator{\rank}{rk}

\newcommand{\rk}{\rank}

\newcommand{\supp}{{\mathrm{supp}}}

\def\B{{\mathcal B}}

\def\pr{{\mathbb P}}

\begin{document}
\title{The $k$-XORSAT threshold revisited}
\author{Amin Coja-Oghlan, Mihyun Kang, Lena Krieg, Maurice Rolvien}
\thanks{Amin Coja-Oghlan is supported by DFG CO 646/3 and DFG CO 646/5.  Mihyun Kang is supported by a Friedrich Wilhelm Bessel research award of the Alexander von Humboldt Foundation (AUT 1204138 BES)}
\address{Amin Coja-Oghlan, {\tt amin.coja-oghlan@tu-dortmund.de}, TU Dortmund, Faculty of Computer Science, 12 Otto-Hahn-St, Dortmund 44227, Germany.}
\address{Mihyun Kang, {\tt kang@math.tu-graz.at}, TU Graz, Institute of Discrete Mathematics, Steyrergasse 30, 8010 Graz, Austria.}
\address{Lena Krieg, {\tt lena.krieg@tu-dortmund.de}, TU Dortmund, Faculty of Computer Science, 12 Otto-Hahn-St, Dortmund 44227, Germany.}
\address{Maurice Rolvien, {\tt maurice.rolvien@tu-dortmund.de}, TU Dortmund, Faculty of Computer Science, 12 Otto-Hahn-St, Dortmund 44227, Germany.}

\maketitle

\begin{abstract}
We provide a simplified proof of the random $k$-XORSAT satisfiability threshold theorem.
As an extension we also determine the full rank threshold for sparse random matrices over finite fields with precisely $k$ non-zero entries per row.
This result is an extension of a result from [Ayre, Coja-Oghlan, Gao, M\"uller: Combinatorica 2020].
The proof combines physics-inspired message passing arguments with a surgical moment computation.
\hfill MSc: 60B20, 15B52
\end{abstract}

\section{Introduction}\label{Sec_intro}

\noindent
The random 3-XORSAT problem was one of the first random constraint satisfaction problems whose satisfiability threshold could be pinpointed precisely.
A random 3-XORSAT instance consists of a conjunction of XOR-clauses, rather than OR-clauses as in the common $k$-SAT problem.
The goal is to find the maximum number of random XOR-clauses such that the formula remains satisfiable with high probability.
The seminal article of Dubois and Mandler~\cite{DuboisMandler} that first solved this problem introduced an influential technique, namely the second moment method applied to a pruned problem instance.
In their very last sentence Dubois and Mandler asserted that their proof extends to $k$-XORSAT for any $k\geq3$.
However, because of the analytic difficulties associated with estimating the second moment for $k>3$, this generalisation turned out to be far from straightforward.
The first complete proof, covering over 30 pages and involving an (avoidable) bit of computer assistance, was published by Pittel and Sorkin~\cite{PittelSorkin} more than a decade later.
Subsequently a different but still fairly complicated proof that relies on coupling arguments rather than moment calculations was suggested by Ayre, Coja-Oghlan, Gao and M\"uller~\cite{Ayre}.
That result covers not only $k$-XORSAT but also an extension to random matrices over finite fields.

The present contribution develops a relatively short, self-contained derivation of the $k$-XORSAT threshold as well as said extensions to random matrices via a novel approach that differs significantly from both~\cite{Ayre,PittelSorkin}.
The new proof is partly inspired by statistical physics ideas and by recent work on a vaguely related random matrix problem~\cite{parity,MM}.
To elaborate, we first derive a quantitative characterisation of a typical solution to a random $k$-XORSAT formula by means of what physicists would call a `quenched' argument.
The quenched argument employs Warning Propagation (`WP'), a physics-inspired message passing technique.
Then we follow up with a surgical moment computation confined to scenarios that match the precise characteristics predicted by WP.
In physics jargon this second bit amounts to an `annealed' computation.
Usually annealed estimates fail to be tight due to large deviations effects.
They also tend to be painfully intricate.
But because the present specimen carefully homes in on solutions with the correct `quenched' properties, the calculations are tight as well as elegant.

Let $\vF=\vF_k(n,m)$ be a random $k$-XORSAT instance with $n$ Boolean variables and $m$ random XOR-clauses of length $k$.
To be precise, the clauses are drawn independently and uniformly from the set of all possible $2^k\binom nk$ XOR-clauses on the variable set $x_1,\ldots,x_n$.
The following theorem, first established in~\cite{DuboisMandler} for $k=3$ and in~\cite{PittelSorkin} for $k>3$, provides the $k$-XORSAT threshold.

\begin{theorem}\label{thm_xor}
	For $k\geq3$ and $d>0$ let
	\begin{align}\label{eqPhi}
		\Phi_{d,k}(\alpha)&=\exp\bc{-d \alpha^{k-1}}+d\alpha^{k-1}-\frac{d(k-1)}{k}\alpha^k-\frac dk&\mbox{and}&&d_k&=\sup\cbc{d>0:\max_{\alpha\in[0,1]}\Phi_{d,k}(\alpha)=1-d/k}.
	\end{align}
	For any $\eps>0$ \whp\ the random $k$-XORSAT formula $\vF$ is
	\begin{enumerate}[(i)]
		\item satisfiable if $m\leq(1-\eps)d_k n/k$,
		\item unsatisfiable if $m\geq(1+\eps)d_k n/k$.
	\end{enumerate}
\end{theorem}

In a nutshell, the $k$-XORSAT satisfiability threshold equals $d_k/k$.
The threshold admits an explicit combinatorial interpretation, an observation that was a vital to the original derivations~\cite{DuboisMandler,PittelSorkin}.
To elaborate, we rephrase the $k$-XORSAT formula $\vF$ as a linear system over $\field_2$ as follows.
Set up a random $m\times n$-matrix $\vA$ whose $i$-th row has one-entries in precisely the $k$ columns $j$ such that variable $x_j$ appears in the $i$-th clause of $\vF$.
Thus, each row of $\vA$ represents a clause.
Further, define $\vy_i=1$ iff $k$ plus the number of negations in the $i$-th clause is odd.
Then every solution $\sigma\in\field_2^n$ to the linear system $\vA\sigma=\vy$ renders a XOR-satisfying assignment of $\vF$, and vice versa.
Because the signs of the literals are independent of the identities of the underlying variables, the vector $\vy$ is independent of $\vA$.
Therefore, the random XOR-formula $\vF$ is satisfiable \whp\ iff $\vA$ has full row rank $m$ \whp

Now consider the following process that prunes $\vA$ down to a minor $\vA^{(2)}$:
\begin{quote}
	while there exists a column with at most a single non-zero entry, remove that column along with the row where its non-zero entry appears (if there is one).
\end{quote}
This is just the random hypergraph 2-core peeling process phrased in terms of the matrix $\vA$.
Therefore, it is possible (albeit non-trivial) to track the pruning process so as to determine the likely size of $\vA^{(2)}$~\cite{Mike}.
This analysis evinces that $d_k n/k$ marks the threshold beyond which $\vA^{(2)}$ has more rows than columns \whp\
In effect, for $m\geq(1+\eps)d_k n/k$ the minor $\vA^{(2)}$ cannot have full row rank anymore, nor can the original matrix $\vA$.
Consequently, $\vF$ is unsatisfiable for $m\geq(1+\eps)d_k n/k$.

Although for $m\leq(1-\eps)d_k n/k$ the minor $\vA^{(2)}$ has fewer rows than columns, it is by no means a foregone conclusion that $\vA^{(2)}$ also has full row rank \whp\
Indeed, in~\cite{DuboisMandler,PittelSorkin} the main technical difficulty lies in demonstrating this fact via the second moment method.
The necessary calculations turn out to be delicate because they operate with the outcome $\vA^{(2)}$ of the pruning process, a matrix whose rows are stochastically dependent.
The second moment therefore involves subtle large deviations trade-offs.
Luckily, the proof strategy that we propose here requires neither an explicit analysis of the pruning process, nor complicated large deviations arguments.

\Thm~\ref{thm_xor} admits a natural generalisation to matrices over finite fields beyond $\field_2$.
Let $q\geq2$ be a prime power and let $\fA=(\fA_{ij})_{i,j\geq1}$ be an infinite matrix with entries $\fA_{ij}\in\field_q\setminus\{0\}$.
Further, given integers $m,n>0$ and $k\geq3$ let $(\vec e_i)_{i\geq1}$ be a family of independent uniformly random subsets of $[n]$ of size $|e_i|=k$ and define a random $m\times n$-matrix $\vA=\vA(k,m,n,q,\fA)$ over $\field_q$ by letting
\begin{align}\label{eqA}
		\vA_{ij}=\fA_{ij}\vecone\{j\in\vec e_i\}&&(i\in[m],\,j\in[n]).
	\end{align}
Thus, $\vA$ has precisely $k$ non-zero entries per row.
The positions of the non-zero entries are determined by the $\ve_i$, while the entries themselves are copied from $\fA$.
Naturally, in the case $q=2$ we simply obtain the matrix induced by the $k$-XORSAT formula $\vF$.
Therefore, the following theorem encompasses \Thm~\ref{thm_xor} as a special case.

\begin{theorem}\label{thm_main}
	For any $k\geq3$, any prime power $q\geq2$ and any infinite matrix $\fA$ composed of non-zero elements of $\field_q$ the following is true.
	Let $d_k$ be the threshold from \eqref{eqPhi}.
	Then for any $\eps>0$,
	\begin{enumerate}[(i)]
		\item if $m\leq(1-\eps)d_k n/k$, then $\vA$ has full row rank \whp
		\item if $m\geq(1+\eps)d_k n/k$, then $\vA$ fails to have full row rank \whp
	\end{enumerate}
\end{theorem}

\Thm~\ref{thm_main} complements~\cite[\Thm~1.1]{Ayre}, where only random matrices with identically distributed rows were considered.
By contrast, in \Thm~\ref{thm_main} the matrix $\fA$ may proscribe different non-zero entries for each row.
That said, in hindsight the theorem shows that the full rank threshold is independent of both $q$ and $\fA$.
We proceed to outline the strategy upon which the proof of \Thm~\ref{thm_main} is based.

\section{Proof strategy}\label{sec_outline}

\noindent
The main difficulty lies in proving the positive statement \Thm~\ref{thm_main}(i).
Suppose we could argue that for $m<(1-\eps)d_k n/k$ \whp\ a random vector $\vec\sigma\in\ker\vA$ is approximately `balanced' in the sense that every value $s\in\FF_q$ appears in $\vec\sigma$ about $n/q$ times. 
Since a straightforward moment calculation shows that the expected number of balanced $\sigma\in\ker\vA$ equals $(1+o(1))q^{n-m}$, we could then conclude that $|\ker\vA|=(1+o(1))q^{n-m}$ \whp, and thus that $\vA$ has full row rank \whp

However, we will not be able to prove directly that a random $\vec\sigma\in\ker\vA$ is balanced \whp\
Instead we will work with a matrix $\vhA$ obtained from $\vA$ by a small but consequential perturbation called `pinning'.
The matrix $\vhA$ contains $\vA$ as its top $m\times n$-minor, but $\vhA$ has $O(\log n)$ additional rows.
Pinning guarantees that $\vhA$ has only relatively few `short linear relations', a property that will pave the way for us to bring the Warning Propagation (`WP') message passing scheme to bear. 
Ultimately we will argue that random $\vhs\in\ker\vhA$ are balanced \whp\
As outlined in the previous paragraph, this will imply that $\vhA$ has full row rank \whp, whence the same is true of $\vA$.

The purpose of WP is to show that the vectors in the kernel of $\vhA$ have a peculiar structure.
Specifically, there are certain coordinates $j\in[n]$ that are `frozen' in $\vhA$, meaning that $\sigma_j=0$ for all $\sigma\in\ker\vhA$.
By contrast, the values assigned to the unfrozen coordinates are essentially balanced.
Hence, if $\vha n$ variables are frozen, then in a random $\vhs\in\ker\vhA$ each non-zero value $s\in\FF_q\setminus\{0\}$ appears about $(1-\vha)n/q$ times.
Ultimately we will argue that $\vha=o(1)$ \whp, which implies that $\vhs$ is balanced \whp\

But the proof that $\vha=o(1)$ \whp\ requires a few more steps.
First, from WP we learn that the probability that $j\in[n]$ is frozen depends on the number of non-zero entries in the $j$-th column of $\vhA$.
In fact, WP renders detailed `local' information about the distribution of the frozen coordinates.
In the quenched part of the analysis, we will extract this information carefully to obtain a quantitative picture of the structure of the kernel vectors in terms of the as yet unknown value of $\vha$.
Moreover, we will see that the messages exchanged by WP satisfy a certain fixed point property.

Subsequently, we will develop an `annealed' (moment computation) argument that allows us to bound the number of WP fixed points associated with any conceivable value of $\vha$.
Moreover, we will compute the expected number of vectors $\sigma\in\ker\vhA$ that are consistent with a given WP fixed point.
This calculation will reveal that \whp\ for $m<(1-\eps)d_kn/k$ no WP fixed point with $\Omega(n)$ frozen coordinates gives rise to $q^{n-m-o(n)}$ kernel vectors, the number of vectors that we know the kernel of $\vhA$ must contain because its rank and its nullity sum to $n$.
Hence, we deduce that $\vha=o(1)$ \whp, as desired.

In the rest of this section we discuss in more detail the proof of \Thm~\ref{thm_main}(i).
We begin with the pinning operation in \Sec~\ref{sec_pin}, then discuss WP and the quenched and annealed analyses.
The proof of the second assertion \Thm~\ref{thm_main} (ii) is but an afterthought.
Indeed, as mentioned in \Sec~\ref{Sec_intro} this second assertion could be derived from known results about the size of the minor $\vA^{(2)}$.
Nonetheless, \Sec~\ref{sec_finish} contains a self-contained proof based on the interpolation method that avoids the analysis of the pruning process.

\subsection{Pinning}\label{sec_pin}
Adding a few rows to a matrix, the randomised pinning operation mostly removes `short linear relations'.
The operation, devised in this form in~\cite{Maurice}, actually works on any matrix, not just on the random matrix $\vA$.
Hence, let $A$ be any $\FF_q$-matrix of size $M\times N$.
For an integer $t\geq0$ obtain $A[t]$ from $A$ by adding $t$ new rows that each contain a single non-zero entry, namely a one in a random position chosen independently and uniformly from the $N$ columns.

The purpose of this operation is to diminish the number of short relations.
To be precise, following~\cite{Maurice} we call a set $J\subset[N]$ of columns a {\em relation of $A$} if there exists a vector $y\in\field_q^M$ such that
	$$\supp(y^\trans A)=\cbc{j\in[N]:(y^\trans A)_j\neq0}$$ 
is a non-empty subset of $J$.
In other words, the non-zero entries of the linear combination $y^\trans A\neq0$ of the rows of $A$ are confined to $J$.
Further, call $j\in[N]$ {\em frozen in $A$} if the singleton $\{j\}$ is a relation of $A$.
Thus, $j$ is frozen iff $\sigma_j=0$ for every $\sigma\in\ker A$.
Let $\cF(A)$ be the set of all frozen $j\in[N]$.

In addition, call $J\neq\emptyset$ a {\em proper relation} of $A$ if $J\setminus\cF(A)$ is a relation of $A$.
Finally, we say that $A$ is {\em $(\delta,\ell)$-free} if $A$ possesses fewer than $\delta\binom{N}{h}$ proper relations $I$ of size $|I|=h$ for any $2\leq h\leq\ell$.
This definition is meant to express that $A$ contains few relations of size $\ell$ that are not `just' composed of frozen $j\in[N]$.%
	\footnote{\Lem~\ref{lem_pinning} is the only statement beyond textbook knowledge that we apply without a proof in order to derive \Thm~\ref{thm_main}. The proof, which relies on a potential function argument and a bit of linear algebra, is neither long nor difficult.}

\begin{lemma}[{\cite[\Prop~2.4]{Maurice}}]\label{lem_pinning}
	For any $\delta>0,\ell>0$ there exists $T_0=O(\ell^3/\delta^4)>0$ such that for any $T\geq T_0$ and any matrix $A$ for a random $\vec t\in[T]$ we have $\pr\brk{A[\vt]\mbox{ is $(\delta,\ell)$-free}}>1-\delta$.
\end{lemma}

Setting $T=\lceil\log n\rceil$, we let $\vhA=\vA[\vt]$ for a random $\vec t\in[T]$.
Since $T_0$ in \Lem~\ref{lem_pinning} is independent of the size of $A$ and scales polynomially in $\ell,\delta$, we obtain the following.

\begin{corollary}\label{cor_vhA}
	Let $\omega=\lceil\log\log n\rceil$.
	\Whp\ $\vhA$ is $(\omega^{-1},\omega)$-free.
\end{corollary}

Thanks to the scarcity of short proper relations provided by \Cor~\ref{cor_vhA} we will be able to characterise the frozen set $\cF(\vhA)$ in terms of the WP message passing scheme, which is the next item on our agenda.

\subsection{Warning Propagation}\label{sec_wp_intro}
Since we will need to work not just with $\vhA$ but also with a few other matrices derived from it, we introduce WP for a general matrix $A$ of size $M\times N$.
The matrix $A$ naturally induces a bipartite graph $G(A)$ called the {\em Tanner graph}.
Its vertex set comprises a set $V_N=\{v_1,\ldots,v_N\}$ of {\em variable nodes} and another set $F_M=\{a_1,\ldots,a_M\}$ of {\em check nodes}.
The former represent the columns of $A$ and the latter the rows.
An edge $a_iv_j$ is present in $G(A)$ iff $A_{ij}\neq0$.
For a vertex $u\in V_N\cup F_M$ let $\partial u=\partial_Au$ denote its set of neighbours.
Moreover, for a set $S\subset V_N\cup F_M$ let $A\setminus S$ be the minor of $A$ obtained by deleting all rows $i$ such that $a_i\in S$ as well as all columns $j$ such that $v_j\in S$.
In defining the WP scheme we follow~\cite{parity}.

The thrust of WP is to characterise the set $\cF(A)$ of frozen variables in terms of just the immediate local interactions between variables and their adjacent checks.
To this end we associate messages with the edges of the Tanner graph.
Specifically, each edge $v_ja_i$ of $G(A)$ comes with one message directed from $v_j$ to $a_i$ and a message in the reverse direction.
The messages take the symbolic values $\{\unfrozen,\frozen\}$ to represent `unfrozen' and `frozen'.
Let
	\begin{align*}
		\fM(A)&=\cbc{\fm=(\fm_{v\to a},\fm_{a\to v})_{v\in V_N,v\in\partial_A a}:\fm_{v\to a},\fm_{a\to v}\in\{\unfrozen,\frozen\}}
	\end{align*}
be the set of all possible collections of messages.
Further, define the {\em standard messages} of $A$ by letting
\begin{align}\label{eqWPdef}
	\fm_{v_j\to a_i}(A)&=\begin{cases} \frozen&\mbox{ if $j\in\cF(A\setminus\cbc{a_i})$}\\ \unfrozen&\mbox{ otherwise}
	\end{cases} &
		\fm_{a_i\to v_j}(A)&=\begin{cases} \frozen&\mbox{ if $v_j\in\cF(A\setminus(\partial v_j\setminus\cbc{a_i}))$}\\ \unfrozen&\mbox{ otherwise}
		\end{cases}\qquad(i\in[M],\,j\in[N]).
\end{align}
Thus, $\fm_{v_j\to a_i}(A)=\frozen$ indicates that variable $v_j$ is frozen in the matrix obtained from $A$ by deleting row $a_i$.
Similarly, $\fm_{a_i\to v_j}(A)=\frozen$ if variable $v_j$ is frozen in the matrix obtained from $A$ by deleting all rows $a_h\in\partial_A v_j$ except for $a_i$.

If indeed freezing were a perfectly local phenomenon transmitted along the edges of the Tanner graph, then the messages \eqref{eqWPdef} should remain invariant under the {\em Warning Propagation update} $\WP_A:\fM(A)\to\fM(A)$, $\fm=(\fm_{\nix\to\nix})\mapsto\WP(\fm)=(\hat\fm_{\nix\to\nix})$, which is defined by
\begin{align}\label{eqWPupdate1}
	\hat\fm_{v_j\to a_i}&=\begin{cases} \frozen&\mbox{ if $\exists a_h\in\partial v_j\setminus\cbc{a_i}:\fm_{a_h\to v_j}=\frozen$},\\
		 \unfrozen&\mbox{ otherwise,}
	 \end{cases} &
		\hat\fm_{a_i\to v_j}&=\begin{cases} \frozen&\mbox{ if $\forall x_h\in\partial a_i\setminus\{v_j\}:\fm_{x_h\to a_i}=\frozen$},\\
			\unfrozen&\mbox{ otherwise.}
	\end{cases}
\end{align}
Indeed, the first update rule $\hat\fm_{v_j\to a_i}$ expresses that we expect $v_j$ to be frozen in $A\setminus\{a_i\}$ iff some other check $a_h$ `freezes' $v_j$.
Similarly, one might expect that $a_i$ causes $v_j$ to freeze iff all the other variables $x_h$ adjacent to $a_i$ freeze, thereby leaving no other option to satisfy $a_i$ but to always set $x_j$ to zero as well.

Finally, in order to extract the set of frozen variables from the WP messages, we define $\{\unfrozen,\slush,\frozen\}$-valued labels to go with the variable and check nodes: for $\fm\in\fM(A)$ let
\begin{align}\label{eqWPmarks1}
	\fm_{v_j}&=\begin{cases} \frozen&\mbox{if $\fm_{a\to v_j}=\frozen$ for at least two $a\in\partial v_j$},\\
		\slush&\mbox{if $\fm_{a\to v_j}=\frozen$ for precisely one $a\in\partial v_j$},\\
		 \unfrozen&\mbox{otherwise,}
	\end{cases} \\
		\fm_{a_i}&=\begin{cases} \frozen&\mbox{if $\fm_{v\to a_i}=\frozen$ for all $v\in\partial a_i$},\\
			\slush&\mbox{if $\fm_{v\to a_i}=\frozen$ for all but one $v\in\partial a_i$},\\
			\unfrozen&\mbox{otherwise.}
	\end{cases}
\label{eqWPmarks2}
\end{align}
Here the new label $\fm_{v_j}(A)=\slush$ (`slush') indicates that $v_j$ is `barely' frozen as there is only one incoming $\frozen$-message.
At first glance the $\slush$-label may seem superfluous as it could just be subsumed by $\frozen$ in the case of $\fm_{v_j}$, and by $\unfrozen$ in $\fm_{a_i}$.
However, under \eqref{eqWPupdate1} the $\slush$-labeled vertices `return' different messages than those labeled $\frozen$ or $\unfrozen$.
For instance, if $\fm_{v_j}=\slush$, then $\hat\fm_{v_j\to a_i}=\unfrozen$ if $\fm_{a_i\to v_j}=\frozen$, whereas in the case $\fm_{v_j}=\frozen$ we have $\hat\fm_{v_j\to a_i}=\frozen$ for all $a_i\in\partial v_j$.
Let $\fm_{v_j}(A)$, $\fm_{a_i}(A)$ denote the labels extracted via \eqref{eqWPmarks1}--\eqref{eqWPmarks2} from the standard messages $\fm_{\nix\to\nix}(A)$ from \eqref{eqWPdef}.

It is easily verified that the WP messages \eqref{eqWPdef} coincide with the updated messages if $G(A)$ is acyclic, i.e.,
\begin{align}\label{eqWPfixedEx}
	\fm_{v_j\to a_i}(A)&=\hat\fm_{v_j\to a_i}(A),&\fm_{a_i\to v_j}(A)&=\hat\fm_{a_i\to v_j}(A),
\end{align}
for all $i,j$ such that $a_i\in\partial_Av_j$. 
But it is equally easy to come up with cyclic Tanner graphs where \eqref{eqWPfixedEx} is violated.

Nonetheless, the following proposition shows that \eqref{eqWPfixedEx} is satisfied on the random matrix $\vhA$ for all but $o(n)$ adjacent pairs $a_i,v_j$ \whp\
The proposition also shows that the labels extracted via \eqref{eqWPmarks1} correctly identify the set $\cF(\vhA)$, up to at most $o(n)$ exceptions.
Furthermore, in most kernel vectors the values of the unfrozen variables are about `balanced'.
Let $\vha=|\cF(\vhA)|/n$ be the fraction of frozen variables of $\vhA$ and let $\vhs$ be a uniformly random element of $\ker\vhA$.
Moreover, let $d_{\vhA}(v_j)$ denote the degree of a variable node $v_j$ in $G(\vhA)$.

\begin{proposition}\label{prop_vhA}
	Let $d>0,k\geq3$.
	\Whp\ we have
	\begin{align}\label{eqprop_vhAi}
		\sum_{i=1}^m\sum_{v_j\in\partial_{\vhA} a_i}\vecone\cbc{\fm_{v_j\to a_i}(\vhA)\neq\hat\fm_{v_j\to a_i}(\vhA)}+\vecone\cbc{\fm_{a_i\to v_j}(\vhA)\neq\hat\fm_{a_i\to v_j}(\vhA)}&=o(n),\\
		\abs{\cbc{j\in[n]:\fm_{v_j}(\vhA)\neq\unfrozen}\triangle\cF(\vhA)}&=o(n)\label{eqprop_vhAii},\\
		\sum_{s\in\FF_q}\sum_{\ell\geq0}\abs{\sum_{j=1}^n\vecone\{d_{\vhA}(v_j)=\ell,\,\fm_{v_j}(\vhA)=\unfrozen\}\bc{\vecone\{\vhs_j=s\}-1/q}}&=o(n).\label{eqprop_vhAiii}
	\end{align}
\end{proposition}

Observe that \eqref{eqprop_vhAiii} posits that the unfrozen variables are not just `balanced' overall (in the sense that every value $s\in\FF_q$ occurs with frequency about $1/q$), but that balance even holds once we break things down to unfrozen variables of some specific degree $\ell\geq0$.
The proof Proposition~\ref{prop_vhA}, which we carry out in \Sec~\ref{sec_standard}, rests on the scarcity of short linear relations provided by \Cor~\ref{cor_vhA}.

\subsection{Quenched analysis}
Recall that our goal is to show that $\vhs$ is approximately balanced \whp\
\Prop~\ref{prop_vhA} reduces this task to showing that $\vha=o(1)$ \whp\
To this end we are going to extract some more detailed information about the WP messages that the variable and check nodes exchange.
Specifically, we are going to estimate the number of variables/checks with specific labels according to \eqref{eqWPmarks1}--\eqref{eqWPmarks2}.
In fact, we even need to know the number of variables/checks with specific labels and with specific numbers of incoming/outgoing message pairs.
Hence, our next goal is to derive such formulas in terms of the (as yet) unknown random variable $\vha$.

To account for the numbers of message pairs received/sent by the various nodes let $\cL$ be the set of all vectors $\ell=(\ell_{\unfrozen \unfrozen},\ell_{\unfrozen \frozen},\ell_{\frozen \unfrozen},\ell_{\frozen\frozen})\in\ZZpos^4$.
For $\ell\in\cL$, a label $z\in\{\frozen,\slush,\unfrozen\}$, a matrix $A$ and a collection of messages $\fm_{\nix\to\nix}\in\fM(A)$ let
\begin{align}\label{eqbiDelta1}
	\Delta_{z,\ell}(\fm_{\nix\to\nix})=\cbc{v\in V(A):\bc{\fm_v=z}\wedge\bc{\forall{s,t\in\{\unfrozen,\frozen\}}:\abs{\cbc{a\in\partial v:\fm_{a\to v}=s,\,\fm_{v\to a}=t}}=\ell_{st}}},\\
		\Gamma_{z,\ell}(\fm_{\nix\to\nix})=\cbc{a\in F(A):\bc{\fm_a=z}\wedge\bc{\forall{s,t\in\{\unfrozen,\frozen\}}:\abs{\cbc{v\in\partial a:\fm_{v\to a}=s,\,\fm_{a\to v}=s}}=\ell_{st}}}.\label{eqbiDelta2}
\end{align}
Thus, $\Delta_{z,\ell}$ comprises variable nodes labelled $z$ by \eqref{eqWPmarks1} that receive/send out numbers of WP messages as detailed by $\ell$.
To be precise, the first label $s$ of $\ell_{st}$ encodes the incoming message, while the second index $t$ specifies the outgoing messages.
Similarly, $\Gamma_{z,\ell}$ counts checks with a given label and given message statistics.

We are going to calculate $|\Delta_{z,\ell}(\fm_{\nix\to\nix}(\vhA))|,|\Gamma_{z,\ell}(\fm_{\nix\to\nix}(\vhA))|$ in terms of the fraction $\vha$ of frozen variables.
As a first step, the following sets comprise the conceivable vectors $\ell$ to go with the various types of variable/check nodes, in line with \eqref{eqWPmarks1}--\eqref{eqWPmarks2}:
\begin{align}
	\cD(\unfrozen)&=\cbc{\ell\in\cL:\ell_{\frozen\unfrozen}=\ell_{\unfrozen \frozen}=\ell_{\frozen \frozen}=0},&
	\cG(\unfrozen)&=\cbc{\ell\in\cL:\ell_{\unfrozen\frozen}=\ell_{\frozen\frozen}=0,\ell_{\unfrozen\unfrozen}\geq2,\ell_{\frozen\unfrozen}=k-\ell_{\unfrozen\unfrozen}},\label{eqcDcG1}\\
	\cD(\slush)&=\cbc{\ell\in\cL:\ell_{\frozen\unfrozen}=1,\ell_{\frozen \frozen}=\ell_{\unfrozen\unfrozen}=0},&
	\cG(\slush)&=\cbc{\ell\in\cL:\ell_{\unfrozen\unfrozen}=\ell_{\frozen \frozen}=0,\ell_{\unfrozen\frozen}=1,\ell_{\frozen\unfrozen}=k-1},\label{eqcDcG2}\\
	\cD(\frozen)&=\cbc{\ell\in\cL:\ell_{\unfrozen\unfrozen}=\ell_{\frozen\unfrozen}=0,\ell_{\frozen\frozen}\geq2},&
	\cG(\frozen)&=\cbc{\ell\in\cL:\ell_{\unfrozen\unfrozen}=\ell_{\unfrozen\frozen}=\ell_{\frozen\unfrozen}=0,\ell_{\frozen\frozen}=k}.\label{eqcDcG3}
\end{align}

Further, we hypothesise that the incoming messages at a check node $a_i$ are essentially independent.
This seems plausible as the Tanner graph $G(\vhA)$ is a sparse random graph with bounded average degree $d$ on the variable side and constant degree $k$ on the check side.
Therefore, typically the neighbouring variable nodes $\partial_{\vhA}a_i$ should end up being far from each other in $G(\vhA\setminus\{a_i\})$, and far apart vertices might conceivably decorrelate.
By a similar token, we expect that the messages received by a typical variable node $v_j$ ought to be nearly independent.
If so, and if we presume that variable-to-check messages take the value $\frozen$ with some probability $0\leq\alpha\leq1$, then check-to-variable messages should take the value $\frozen$ with probability $\alpha^{k-1}$;
for according to \eqref{eqWPupdate1} a check-to-variable message should be $\frozen$ iff all of the check's other $k-1$ incoming messages are $\frozen$.
In light of \eqref{eqcDcG1}--\eqref{eqcDcG3} we can thus predict the frequencies for the variable/check nodes of the various types.
For instance, if $\vha=\alpha$, then we expect to see about $\bar\delta(\alpha,\unfrozen)=\exp(-d\alpha^{k-1})n$ variables $v_j$ with $\fm_{v_j}(\vhA)=\frozen$.
This is because by \eqref{eqcDcG1} such a variable $v_j$ must not receive any $\frozen$-messages, while the mean number of such incoming messages should be $d\alpha^{k-1}$.
Similarly, we arrive at predictions for the frequencies of the other node types:
\begin{align}
	\bar\delta(\alpha,\unfrozen)&=\exp(-d\alpha^{k-1}),&
	\bar\delta(\alpha,\slush)&=d\alpha^{k-1}\exp(-d\alpha^{k-1}),&
	\bar\delta(\alpha,\frozen)&=1-\exp(-d\alpha^{k-1})(1+d\alpha^{k-1}),\label{eqdeltagamma1}\\
	\bar\gamma(\alpha,\unfrozen)&=1-k(1-\alpha)\alpha^{k-1}-\alpha^k,&
	\bar\gamma(\alpha,\slush)&=k(1-\alpha)\alpha^{k-1},&
	\bar\gamma(\alpha,\frozen)&=\alpha^k.\label{eqdeltagamma2}
\end{align}

Finally, extending the reasoning outlined in the previous paragraph, we can derive predictions as to the frequencies of nodes with various labels and given statistics $\ell\in\cL$ of incoming/outgoing messages.
With $\Po_{\geq 2}(\lambda)$ and $\Bin_{\geq2}(N,p)$ denoting the conditional Poisson/Binomial distributions given an outcome of at least two, we obtain the following expressions: 
\begin{align}\label{eqDeltafell}
	\bar\Delta_{\unfrozen,\ell}(\alpha)&=\bar \delta(\alpha,\unfrozen)\vecone\{\ell\in\cD(\unfrozen)\}\pr[\Po(d(1-\alpha^{k-1}))=\ell_{\unfrozen\unfrozen}],\\
	\bar\Delta_{\slush,\ell}(\alpha)&=\bar \delta(\alpha,\slush)\vecone\{\ell\in\cD(\slush)\}\pr[\Po(d(1-\alpha^{k-1}))=\ell_{\unfrozen\frozen}],\\
	\bar\Delta_{\frozen,\ell}(\alpha)&=\bar \delta(\alpha,\frozen)\vecone\{\ell\in\cD(\frozen)\}\pr[\Po_{\geq2}(d\alpha^{k-1})=\ell_{\frozen\frozen}]\pr[\Po(d(1-\alpha^{k-1}))=\ell_{\unfrozen\frozen})],\\
	\bar\Gamma_{\unfrozen,\ell}(\alpha)&=\bar \gamma(\alpha,\unfrozen)\vecone\{\ell\in\cG(\unfrozen)\}\pr\brk{\Bin_{\geq2}(k,1-\alpha)=\ell_{\unfrozen\unfrozen}}\label{eqGammafell0},
	\\
	\bar\Gamma_{\slush,\ell}(\alpha)&=\bar \gamma(\alpha,\slush)\vecone\{\ell\in\cG(\slush)\},\\
	\bar\Gamma_{\frozen,\ell}(\alpha)&=\bar \gamma(\alpha,\frozen)\vecone\{\ell\in\cG(\frozen)\}.\label{eqGammafell}
\end{align}
The following proposition shows that the aforementioned predictions are accurate \whp

\begin{proposition}\label{prop_wp}
	Let $d>0,k\geq3$.
	Then
	\begin{align*}
		\sum_{z\in\{\frozen,\slush,\unfrozen\}}\sum_{\ell\in\cL}\ex\abs{|\Delta_{z,\ell}(\fm_{\nix\to\nix}(\vhA))|-n\bar\Delta_{z,\ell}(\vha)}+\ex\abs{|\Gamma_{z,\ell}(\fm_{\nix\to\nix}(\vhA))|-m\bar\Gamma_{z,\ell}(\vha)}&=o(n).
	\end{align*}  
\end{proposition}

Thus, $|\Delta_{z,\ell}(\fm_{\nix\to\nix}(\vhA))|$, $|\Gamma_{z,\ell}(\fm_{\nix\to\nix}(\vhA))|$ approximately equal $\bar\Delta_{z,\ell}(\vha)n$, $\bar\Gamma_{z,\ell}(\vha)m$ evaluated at the actual fraction $\vha$ of frozen variables of $\vhA$, which is a random variable.
The proof of \Prop~\ref{prop_wp}, which can be found in \Sec~\ref{sec_prop_wp}, is based on coupling arguments.
In particular, the proof does not reveal the likely value of $\vha$.

\subsection{Annealed arguments}
In light of \eqref{eqprop_vhAiii} from \Prop~\ref{prop_vhA} our main task is to show that $\vha=o(1)$ \whp\ if $d<(1-\eps)d_k/k$.
To this end we are going to combine \Prop~\ref{prop_wp} with a first moment argument that shows that for $d<(1-\eps)d_k/k$ only the scenario $\vha=o(1)$ \whp\ can account for the $q^{n-m+o(n)}$ vectors that the kernel of the $(m+o(n))\times n$-matrix $\vhA$ must inevitably contain.
In other words, we are going to show that WP fixed points with $\Omega(n)$ frozen variables come with too small a number of kernel vectors.

In this respect the present argument differs significantly from prior proofs of \Thm~\ref{thm_xor}~\cite{DuboisMandler,PittelSorkin}.
Instead of first investigating the likely shape of vectors in the kernel (specifically, that they `come from' WP fixed points with certain statistics), these analyses directly estimate the expected number of vectors in the kernel with a given Hamming weight; of course, this kind of argument is workable only in the case $q=2$.
The drawback of a blunt moment computation is that even extremely rare events make a contribution.
Such large deviations effects tend to lead to intricate and technically demanding analytical optimisation problems.

The present `annealed' argument (viz.\ moment computation) consists of two layers.
First we estimate the expected number of WP fixed points with the `correct' statistics as provided by \eqref{eqDeltafell}--\eqref{eqGammafell}.
To be precise, reminding ourselves of the update rules \eqref{eqWPupdate1}, we call $\fm_{\nix\to\nix}\in\fM(\vhA)$ an {\em $\alpha$-WP fixed point} if
\begin{align}\label{eqalphaWP1}
	\sum_{i=1}^m\sum_{v_j\in\partial_{\vhA}a_i}\vecone\cbc{\fm_{v_j\to a_i}\neq\hat\fm_{v_j\to a_i}}+\vecone\cbc{\fm_{a_i\to v_j}\neq\hat\fm_{a_i\to v_j}}&=o(n)\qquad\mbox{and}\\
	\sum_{z\in\{\frozen,\slush,\unfrozen\}}\sum_{\ell\in\cL}\abs{|\Delta_{z,\ell}(\fm)|-n\bar\Delta_{z,\ell}(\alpha)}+\abs{|\Gamma_{z,\ell}(\fm)|-m\bar\Gamma_{z,\ell}(\alpha)}&=o(n).\label{eqalphaWP2}
\end{align}
Thus, we ask that most messages be invariant under the update \eqref{eqWPupdate1}, and that the counts $|\Delta_{z,\ell}(\fm)|,|\Gamma_{z,\ell}(\fm)|$ be in line with \Prop~\ref{prop_wp}.
Performing relatively simple manipulations of the formulas \eqref{eqDeltafell}--\eqref{eqGammafell}, we will ultimately see that the expected number of $\alpha$-WP fixed points is sub-exponential for any $0\leq\alpha\leq1$.

As a next step, we will estimate the number of kernel vectors $\sigma$ that come with a particular WP fixed point.
To be precise, call $\sigma\in\ker\vhA$ an {\em extension of $\fm_{\nix\to\nix}\in\fM(\vhA)$} if 
\begin{align}\label{eqextension}
	\sum_{i=1}^n\vecone\{\fm_{v_i}\neq\unfrozen,\,\sigma_i\neq0\}+\sum_{s\in\FF_q}\sum_{\ell\geq0}\abs{\sum_{i=1}^n\vecone\{d_{\vhA}(v_i)=\ell,\,\fm_{v_i}=\unfrozen\}(\vecone\{\sigma_i=s\}-1/q)}&=o(n).
\end{align}
Thus, $\sigma$ is required to (mostly) respect the variables that $\fm_{\nix\to\nix}$ deems frozen under \eqref{eqWPmarks1} by actually setting them to zero.
Moreover, the variables deemed unfrozen according to $\fm_{\nix\to\nix}$ need to be assigned in a balanced manner, even when broken down to specific values $\ell$ of the variable degree, just like in \eqref{eqprop_vhAiii}.
In fact, \Prop s~\ref{prop_vhA} and~\ref{prop_wp} show that a random kernel vector $\vhs$ is an $\vha$-extension of the standard messages $\fm_{\nix\to\nix}(\vhA)$.
Hence, letting $\vX_\alpha$ be the number of pairs $(\fm_{\nix\to\nix},\sigma)$ such that $\fm_{\nix\to\nix}$ is an $\alpha$-WP fixed point of $\vhA$ and $\sigma$ is an extension of $\alpha$, we see that $|\ker\vhA|\sim\vX_{\vha}$ \whp\
By comparison, the following proposition, which we prove in \Sec~\ref{sec_covers}, provides a first moment upper bound on $\vX_\alpha$ for any $0\leq\alpha\leq1$ in terms of the function $\Phi_{d,k}$ from~\eqref{eqPhi}.

\begin{proposition}\label{prop_annealedwp}
	Let $d>0,k\geq3$.
	\Whp\ for all $\alpha\in[0,1]$ we have $\ex[\vX_\alpha\mid\fD]\leq q^{n\Phi_{d,k}(\alpha)+o(n)}$. 
\end{proposition}

Since for $d<d_k$ the function $\Phi_{d,k}$ attains its unique maximum at $\alpha=0$ and $q^{n\Phi_{d,k}(0)}=q^{n-m}$, it is not very difficult to derive the estimate $\vha=o(1)$ \whp\ from \Prop~\ref{prop_annealedwp}.
From this, in turn, we can deduce that \whp\ most vectors in the kernel are `balanced', i.e., contain every value $s\in\field_q$ with about equal frequency.
To be precise, for a vector $\sigma\in\field_q^n$ let $\rho(\sigma)=(\rho_s(\sigma))_{s\in\FF_q}$ be the vector with entries $\rho_s(\sigma)=\frac1n\sum_{i=1}^n\vecone\{\sigma_i=s\}$.

\begin{corollary}\label{cor_annealedwp}
	Let $\eps>0$.
	If $m<(1-\eps)d_kn/k$, then \whp\ we have
	\begin{align}\label{eqcor_annealedwp}
		\ex\brk{\|\rho(\vhs)-q^{-1}\vecone\|_2\mid\vhA}=o(1).
	\end{align}
\end{corollary}

It is quite easy to calculate the expected number of vectors $\sigma\in\ker\vhA$ with $\|\rho(\sigma)-q^{-1}\vecone\|_2=o(1)$.
Recall that we obtained $\vhA$ from $\vA$ by adding $\vt$ extra rows with a single non-zero entry each.
In \Sec~\ref{sec_lem_annealed_trivial} we prove the following.

\begin{lemma}\label{lem_annealed_trivial}
	For any $d>0,k\geq3$ there is $\eta>0$ such that $\ex\abs{\{\sigma\in\ker\vhA:\|\rho(\sigma)-q^{-1}\vecone\|_2<\eta\}\mid\vt}\leq(1+o(1))q^{n-m-\vt}.$ 
\end{lemma}

\begin{proof}[Proof of \Thm~\ref{thm_main} (i)]
	Since the top $m$ rows of $\vhA$ are equal to $\vA$, it suffices to prove that $\vhA$ has full row rank \whp\
	Hence, let $\vy^\dagger\in\FF_q^{m+\vt}$ be a uniformly random vector that is conditionally independent of $\vhA$ given $\vt$.
	In order to conclude that $\vhA$ has full row rank \whp, we just need to show that 
		\begin{align}\label{eqthm_main_i1}
			\pr\brk{\exists \sigma\in\FF_q^n:\vhA \sigma=\vy^\dagger}\sim1.
		\end{align}

	Let $\vZ$ be the number of solutions $\sigma$ to $\vhA \sigma=\vy^\dagger$.
	Because $\vy^\dagger$ is independent of $\vhA$ given $\vt$, we have
	\begin{align}\label{eqthm_main_i2}
		\ex\brk{\vZ\mid\vt}&=q^{n-m-\vt}.
	\end{align}
	Further, let $\cB$ be the event that $\vhA$ enjoys the property \eqref{eqcor_annealedwp}.
	By \Cor~\ref{cor_annealedwp} the mean of $\vZ$ on $\cB$ comes to
\begin{align}
	\ex\brk{\vZ\cdot\vecone\cB\mid\vt}&=\sum_{\substack{A^\dagger\in\FF_q^{(m+\vt)\times n}\\y^{\dagger}\in\FF_q^{m+\vt},\,\sigma\in\FF_q^n}}\vecone\cbc{A^\dagger \sigma=y^\dagger,\,A^\dagger\in\cB}\pr\brk{\vhA=A^\dagger,\vhy=y^\dagger\mid\vt}=q^{n-m-\vt}\pr\brk{\vhA\in\cB}\sim\ex[\vZ\mid\vt].
\label{eqthm_main_i3}
	\end{align}
	Similarly, \Lem~\ref{lem_annealed_trivial} yields
\begin{align}\nonumber
	\ex\brk{\vZ^2\cdot\vecone\cB\mid\vt}&=\ex\brk{\vZ|\ker\vhA|\cdot\vecone\cB\mid\vt}=\sum_{A^\dagger,y^{\dagger},\sigma}\vecone\cbc{A^\dagger \sigma=y^\dagger,\,A^\dagger\in\cB}|\ker A^\dagger|\pr\brk{\vhA=A^\dagger,\vhy=y^\dagger\mid\vt}\\
										&=q^{n-m-\vt}\ex\brk{|\ker\vhA|\cdot\vecone\cB\mid\vt}\leq(1+o(1))q^{2(n-m-\vt)}.
\label{eqthm_main_i4}
	\end{align}
	Combining \eqref{eqthm_main_i2}--\eqref{eqthm_main_i4} with Chebyshev's inequality, we see that $\vZ\sim q^{n-m-\vt}>0$ \whp, which implies \eqref{eqthm_main_i1}.
\end{proof}

\subsection{Discussion}
Preceding the seminal contribution of Dubois and Mandler~\cite{DuboisMandler} that determined the precise $3$-XORSAT threshold, Creignou, Daud\'e and Dubois~\cite{CDD} obtained upper and lower bounds by means of the first and the second moment methods.
These methods went on to become a mainstay of the theory of random constraint satisfaction problems, with numerous important additions~\cite{ANP,DSS3}.
Independently of~\cite{PittelSorkin}, a rigorous derivation of the $k$-XORSAT threshold for general $k$ was outlined in~\cite{Dietzfelbinger}, where the threshold was needed for an application to cuckoo hashing.
The $k$-XORSAT threshold was further investigated from the viewpoint of the physicists' replica and cavity methods~\cite{MRTZ}.
Moreover, the contributions~\cite{AchlioptasMolloy,Ibrahimi} conduct a detailed study of the geometry of the solution space of random $k$-XORSAT formulas.

Various different analyses of the pruning process have been put forward~\cite{OJJ,Cooper, Fernholz2, JansonLuczak, Kim, Pittel, Riordan}.
The methods employed in these works range from differential equations to branching processes to enumerative arguments.
Since none of the proofs are particularly simple, we consider the fact that, in contrast to~\cite{DuboisMandler,PittelSorkin}, the present derivation of the $k$-XORSAT threshold gets by without an explicit investigation  of the pruning process a significant plus.

The derivation of the full rank threshold~\cite{Ayre} also avoided an analysis of the pruning process and instead relied on the Aizenman-Sims-Starr coupling argument from mathematical physics~\cite{Aizenman}.
The main result of~\cite{Ayre} is a variant of \Thm~\ref{thm_main} with identically distributed rows.
Specifically, the non-zero entries in the rows are drawn independently from a given distribution on $(\FF_q\setminus\{0\})^k$.
The present proof method can be easily adapted to cover this scenario, but also allows for the non-zero entries to be copied from a given infinite matrix $\fA$, in which case the rows need not be identically distributed anymore.
Prior to~\cite{Ayre}, which still covers over 50 pages, an extension of the $k$-XORSAT threshold result to random matrices over $\FF_3$ was obtained~\cite{GoerdtFalke} by a generalisation of the moment method from~\cite{DuboisMandler,PittelSorkin}.
The article of over 80 pages requires computer assistance.

The techniques developed in~\cite{Ayre} were extended to more general random matrix models with identically distributed rows~\cite{COGHKLMR}; the main result of that paper also implies the $k$-XORSAT threshold, but the proof is rather complicated.
Additionally, for a still more general model of random matrices over general (not necessarily finite) fields an asymptotic formula for the normalised rank was obtain via the Aizenman-Sims-Starr scheme~\cite{Maurice}.
Furthermore, an independent result yields the asymptotic rank of the random matrix $\vA$ over $\FF_2$, albeit without obtaining the precise full rank threshold~\cite{CFP}.
Here we employ the pinning technique from~\cite{Maurice} (\Lem~\ref{lem_pinning}), which is an adaptation of the more general pinning method for discrete probability distributions developed in~\cite{Montanari,Raghavendra}.

Finally, a recent article~\cite{parity} studies sparse square random matrices over $\field_2$ with independent entries.
The main results, pertaining to the structure of the kernel of such a random matrices, evince a somewhat remarkable bifurcation that contrasts with the zero-one behaviour otherwise characteristic of probabilistic combinatorics.
In the present paper we employ the mathematical formalisation of the WP message passing scheme developed in~\cite{parity}.
Furthermore, the article~\cite{parity} also employed a moment computation similar to the one that we use to prove \Prop~\ref{prop_annealedwp}, but for a substantially different matrix model and towards a somewhat different overall result (an analysis of the kernel geometry rather than a proof that the matrix has full rank).

\subsection{Preliminaries}\label{sec_pre}
We need to reflect on the function $\Phi_{d,k}$ and its maxima.
Let
\begin{align}\label{eqphi}\phi_{d,k}(\alpha)=1-\exp(-d \alpha^{k-1}).\end{align}
A tiny bit of calculus reveals that the functions $\phi_{d,k}$ from \eqref{eqphi} and $\Phi_{d,k}$ from \eqref{eqPhi} are closely related as
\begin{align}\label{eqPhi'}
	\Phi_{d,k}'(\alpha)&=d(k-1)\alpha^{k-2}\bc{\phi_{d,k}(\alpha)-\alpha},\\
	\Phi_{d,k}''(\alpha)&=d(k-1)(k-2)\alpha^{k-3}\bc{\phi_{d,k}(\alpha)-\alpha}-d(k-1)\alpha^{k-2}\bc{1-\phi_{d,k}'(\alpha)}.\label{eqPhi''}
\end{align}
Thus, the fixed points $\alpha\in[0,1]$ of $\phi_{d,k}$ coincide with the stationary points of $\Phi_{d,k}$. 
In fact, the stable fixed points of $\phi_{d,k}$ are precisely the local maxima of $\Phi_{d,k}$.
Moreover, a few lines of calculus reveal the following.

\begin{fact}\label{fact_phi}
	Let $d>0,k\geq3$.
	The function $\phi_{d,k}$ has at most three distinct fixed points in the unit interval, which we denote by $\alpha_{\unfrozen}(d,k)\leq\alpha_{\slush}(d,k)\leq\alpha_{\frozen}(d,k)$.
	There exists a critical value $0<d_k^*<d_k$ such that
	\begin{itemize}
		\item for $d<d_k^*$ we have $\alpha_{\unfrozen}(d,k)\leq\alpha_{\slush}(d,k)\leq\alpha_{\frozen}(d,k)=0$,
		\item for $d=d_k^*$ we have $0=\alpha_{\unfrozen}(d,k)<\alpha_{\slush}(d,k)=\alpha_{\frozen}(d,k)<1$,
		\item for $d>d_k^*$ we have $0=\alpha_{\unfrozen}(d,k)<\alpha_{\slush}(d,k)<\alpha_{\frozen}(d,k)<1$.
	\end{itemize}
	For $d<d_k$ the function $\Phi_{d,k}$ attains its unique maximum at $0$, while $\alpha_{\frozen}(d,k)$ is the unique maximiser for $d>d_k$.
\end{fact}

Additionally, we need the following elementary fact from linear algebra.

\begin{fact}[{\cite[\Lem~2.5]{Maurice}}]\label{fact_indep}
	Let $A,B,C$ be matrices of sizes $M\times N$, $M'\times N$ and $M'\times N'$, respectively.
	Moreover, let $I\subset[N]$ be the set of non-zero columns of $B$ and obtain $B_0$ from $B$ by replacing for every $i\in I\cap\cF(A)$ the $i$-th column of $B$ by zero.
	Unless $I$ is a proper relation of $A$ we have
	\begin{align*}
		\nul\bc{\begin{array}{cc}A&0\\B&C\end{array}}-\nul A+\rk(B_0\ C)=N'.
	\end{align*}
\end{fact}

Further, we make a note of the degree distribution of the Tanner graph $G(\vhA)$.
Because the rows are chosen independently, the degrees of the variable nodes are asymptotically Poisson.
More precisely, routine arguments show that the following is true.

\begin{fact}\label{fact_degs}
	\Whp\ we have $\sum_{\ell\geq0}\exp(\ell)\abs{\pr\brk{\Po(d)=\ell}-\sum_{i=1}^n\vecone\{d_{\vhA}(v_i)=\ell\}}=o(n).$ 
\end{fact}

For the entropy of a probability distribution $p$ on a finite set $\Omega$ we use the symbol
\begin{align*}
	H(p)&=-\sum_{\omega\in\Omega}p(\omega)\log p(\omega),
\end{align*}
with the convention that $0\log0=0$.
Finally, for a vector $\xi\in\FF_q^N$ we write $\|\xi\|_h$ for the $\ell^h$-norm of $\xi$, with the convention that $\|\xi\|_0=|\supp \xi|=|\{i\in[N]:\xi_i\neq0\}|.$

\section{Warning Propagation}\label{sec_standard}

\noindent
In this section we prove \Prop s~\ref{prop_vhA} and~\ref{prop_wp}.
We begin with some ruminations on short linear relations.

\subsection{Short linear relations}\label{sec_la}
The following lemma shows that if a matrix $A$ possesses few short proper relations, then the same is true of any matrix $A'$ obtained from $A$ by adding a single row.
Moreover, $A$ and $A'$ have more or less the same frozen variables.

\begin{lemma}\label{lem_tinker}
	For any $\delta'>0$, $\ell'\geq2$ there exist $\delta>0,\ell\geq2,N_0>0$ such that for any $N>N_0,M>0$, any $M\times N$-matrix $A$ and any matrix $A'$ obtained from $A$ by adding a single row the following is true.
	If $A$ is $(\delta,\ell)$-free, then
	\begin{enumerate}[(i)]
		\item $A'$ is $(\delta',\ell')$-free, and
		\item $|\cF(A')\setminus\cF(A)|<\delta'N$.
	\end{enumerate}
\end{lemma}
\begin{proof}
	Set $\ell = 2\ell'$ and $\delta = \delta'^22^{-\ell-16}$. 
	Assume for contradiction that $A$ is $(\delta,\ell)$-free but that $A'$ fails to be $(\delta',\ell')$-free.
	Let $\cI'$ be the set of all proper relations $I'$ of $A'$ of size $|I'|=\ell'$ that fail to be proper relations of $A$.
	Since $\cF(A)\subset\cF(A')$, for any $y\in\field_q^{M+1}$ with $$\emptyset\neq\supp (y^\trans A')\subset I'\setminus\cF(A')\subset I'\setminus\cF(A)$$ we have $y_{M+1}\neq0$.
	Furthermore, for sufficiently large $N_0$ the set $\cI'\times\cI'$ contains at least $(\delta'\binom N{\ell'})^2/8$ pairs $(I',I'')$ such that $I'\cap I''=\emptyset$.
	Given such a pair $(I',I'')$ let $y,z\in\FF_q^{M+1}$ be such that $\emptyset\neq\supp(y^\trans A')\subset I'\setminus\cF(A')$ and $\emptyset\neq\supp(z^\trans A')\subset I''\setminus\cF(A')$.
	Since $y_{M+1},z_{M+1}\neq0$, there exists $\zeta\in\FF_q\setminus\{0\}$ such that $y_{M+1}+\zeta z_{M+1}=0$.
	Hence,
	\begin{align*}
		\emptyset\neq\supp(((y_1\cdots y_M)+\zeta(z_1\cdots z_M))^\trans A)\subset (I'\cup I'')\setminus\cF(A)\quad\mbox{and}\quad((y_1\cdots y_M)+\zeta(z_1\cdots z_M))^\trans A\neq0.
	\end{align*}
	Consequently, $I'\cup I''$ is a proper relation of $A$ of size $2\ell'$.
	Thus, $A$ possesses at least $(\delta'\binom N{\ell'})^2/8$ such proper relations.
	However, choosing $N_0$ large enough, we obtain $(\delta'\binom N{\ell'})^2/8>\delta\binom{N}\ell$, in contradiction to the fact that $A$ is $(\delta,\ell)$-free.

	Concerning (ii), let $j,j'\in\cF(A')\setminus\cF(A)$ be two distinct indices that are frozen in $A'$ but not in $A$.
	Then there exist vectors $y,z\in\FF_q^{M+1}$ such that $\supp(y^\trans A')=\{j\}$ and $\supp(z^\trans A')=\{j'\}$.
	Since $j,j'\not\in\cF(A)$ we have $y_{M+1}\neq 0\neq z_{M+1}$.
	Hence, there exists $\zeta\in\FF_q\setminus\{0\}$ such that $y_{M+1}+\zeta z_{M+1}=0$.
	Moreover,
	\begin{align*}
		\supp(((y_1\cdots y_M)+\zeta(z_1\cdots z_M))^\trans A)=\{j,j'\}.
	\end{align*}
	Thus, $\{j,j'\}$ is a proper relation of $A$.
	We therefore conclude that $A$ possesses at least $\binom{|\cF(A')\setminus\cF(A)|}2$ proper relations of size two.
	Consequently, $\binom{|\cF(A')\setminus\cF(A)|}2<\delta\binom N2$, whence the desired bound $|\cF(A')\setminus\cF(A)|<\delta'N$ follows.
\end{proof}

Repeated application of \Lem~\ref{lem_tinker} shows the following.

\begin{corollary}\label{cor_tinker}
	There exists $1\ll \omega'=\omega'_n\ll\omega=\omega_n$ such that the following is true.
	Suppose that $A$ is $(\omega,1/\omega)$-free and that $A'$ is obtained from $A$ by adding at most $\omega'$ rows.
	Then $A'$ is $(\omega',1/\omega')$-free and $|\cF(A')\setminus\cF(A)|\leq n/\omega'$.
\end{corollary}

\subsection{Proof of \Prop s~\ref{prop_vhA}}\label{sec_prop_vhA}
\Prop~\ref{prop_vhA} posits that the standard WP messages from \eqref{eqWPdef} are an approximate fixed point of the update rule \eqref{eqWPupdate1} and that the labels defined in \eqref{eqWPmarks1}--\eqref{eqWPmarks2} match their intended semantics.
The starting point of the proof is that the distribution of the random matrix $\vhA$ remains asymptotically invariant under the following resampling operation.

\begin{fact}\label{fact_astar}
	Let $\vA^+$ be the matrix obtained from $\vhA$ via the following operation.
	\begin{align*}
		\parbox{14cm}{Choose a variable node $\vv\in\{v_1,\ldots,v_n\}$ randomly, then independently for all $a\in\partial_{\vhA}\vv$ resample the neighbours of $a$ other than $\vv$ uniformly without replacement from $\{v_1,\ldots,v_n\}\setminus\{\vv\}$.}
	\end{align*}
	Then $\vhA$ and $\vA^+$ are identically distributed. 
\end{fact}

To establish the fixed point property \eqref{eqprop_vhAi} we are going to show that
\begin{align}\label{eqprop_vhA_ii1}
	\fm_{\vv\to a}(\vA^+)=\hat\fm_{\vv\to a}(\vA^+)\qquad\mbox{for all $a\in\partial_{\vA^+}\vv$ \whp;}
\end{align}
then Markov's inequality implies that $\sum_{j=1}^n\sum_{a\in\partial_{\vhA}v_j}\vecone\{\fm_{v_j\to a}(\vhA)\neq\hat\fm_{v_j\to a}(\vhA)\}=o(n)$ \whp\
More specifically, we are going to exhibit an event $\cE$ with $\pr\brk\cE\sim1$ such that \eqref{eqprop_vhA_ii1} holds on $\cE$ deterministically.

To define the event $\cE$ pick a sequence $\Delta=\Delta(n)\gg1$ that diverges slowly enough as $n\to\infty$.
Moreover, obtain $\vA^-$ from $\vA^+$ by deleting all checks $a\in\partial_{\vA^+}\vv$.
Now, let $\cE$ be the event that the three following conditions hold.
\begin{description}
	\item[E1] The second neighbourhood $\partial_{\vA^+}^2\vv=\{v_j:\exists a\in\partial_{\vA^+}\vv:v_j\in\partial_{\vA^+}a\}\setminus\{\vv\}$ has size precisely $(k-1)|\partial_{\vA^+}\vv|\leq\Delta$.
	\item[E2] $\partial_{\vA^+}^2\vv$ is not a proper relation of $\vA^-$.
	\item[E3] For all $a\in\partial_{\vA^+}\vv$ we have $\cF(\vA^+\setminus\{a\})\cap\partial_{\vA^+}a\setminus\{\vv\}=\cF(\vA^-)\cap\partial_{\vA^+}a\setminus\{\vv\}$.
\end{description}

\begin{claim}\label{claim_E}
	We have $\pr\brk{\cE}=1-o(1)$.
\end{claim}
\begin{proof}
	Condition {\bf E1} asks that $\vv$ have degree at most $\Delta/(k-1)$ and that the subgraph of $G(\vhA)$ induced by the vertices of distance at most two from $\vv$ be acyclic.
	Fact~\ref{fact_degs}, Fact~\ref{fact_astar} and the independence of the positions of the non-zero entries in the different rows of $\vhA$ imply that this is indeed the case \whp\
	Moreover, {\bf E1} and the construction of $\vA^+$ ensure that $\partial_{\vA^+}^2\vv$ is nothing but a random set of variable nodes of $G(\vA^-)$ of size at most $\Delta$.
	Since $\vA^+$ contains the same $\vt$ rows with ones in random positions that we added to $\vhA$ by way of the pinning operation, \Lem~\ref{lem_pinning} shows that $\vA^+$ is $(\omega,1/\omega)$-free with probability $1-o(1/\omega)$ for a certain $\omega\gg1$.
	Consequently, $\partial_{\vA^+}^2\vv$ is not a proper relation of $\vA^+$ \whp, provided that $1\ll\Delta\ll\omega$ diverges sufficiently slowly.
	Hence, {\bf E2} holds \whp\
	Finally, $\vA^+\setminus\{a\}$ is obtained from $\vA^-$ by adding at most $\Delta$ rows.
	Therefore, {\bf E2} and \Cor~\ref{cor_tinker} imply that {\bf E3} is satisfied \whp, once again providing that $\Delta\to\infty$ sufficiently slowly.
\end{proof}

The following two claims deliver \eqref{eqprop_vhA_ii1}.

\begin{claim}\label{claim_mhatm1}
	Assume that $\cE$ occurs and let $a\in\partial_{\vA^+}\vv$.
	If there exists $b\in\partial_{\vA^+}\vv\setminus\{a\}$ such that $\fm_{b\to\vv}(\vA^+)=\frozen$, then $\fm_{\vv\to a}(\vA^+)=\frozen$.
	Moreover, if $\fm_{\vv}(\vA^+)\neq\unfrozen$, then $\vv\in\cF(\vA^+)$.
\end{claim}
\begin{proof}
	Let $b\in\partial_{\vA^+}\vv\setminus\{a\}$ be such that $\fm_{b\to\vv}(\vA^+)=\frozen$.
	Then {\bf E3} guarantees that $y\in\cF(\vA^-)$ for all $y\in\partial_{\vA^+}b\setminus\{\vv\}$.
	Therefore, for all $\sigma\in\ker(\vA^+\setminus\cbc{a})\subset\ker(\vA^-\setminus\{a\})$ and all $y\in\partial_{\vA^+}b\setminus\{\vv\}$ we have $\sigma_y=0$, and consequently $\sigma_{\vv}=0$.
	Hence, $\vv\in\cF(\vA^+\setminus\{a\})$, and thus $\fm_{\vv\to a}(\vA^+)=\hat\fm_{\vv\to a}(\vA^+)=\frozen$ by \eqref{eqWPdef}.
	A similar argument yields the second assertion.
\end{proof}

\begin{claim}\label{claim_mhatm2}
	Assume that $\cE$ occurs and let $a\in\partial_{\vA^+}\vv$.
	If $\fm_{b\to\vv}(\vA^+)=\unfrozen$ for all $b\in\partial_{\vA^+}\vv\setminus\{a\}$, then $\fm_{\vv\to a}(\vA^+)=\unfrozen$.
	Moreover, if $\fm_{\vv}(\vA^+)=\unfrozen$, then $\vv\not\in\cF(\vA^+)$.
\end{claim}
\begin{proof}
With $\Pi,\Pi'$ suitable permutation matrices (to reshuffle the rows and columns appropriately), $B$ a matrix of size $|\partial_{\vA^+}\vv\setminus\{a\}|\times(n-1)$ and $C$ a matrix of size $|\partial_{\vA^+}\vv\setminus\{a\}|\times1$, we can write
\begin{align}\label{eqclaim_mhatm2_0}
	\vA^+\setminus\{a\}&=\Pi\cdot\begin{pmatrix} \vA^-\setminus\{\vv\}&0\\ B&C \end{pmatrix}\cdot\Pi'.
\end{align}
Here the submatrix $(B\ C)$ corresponds to the checks $b\in\partial_{\vA^+}\vv\setminus\{a\}$, and the last column $\binom 0C$ represents $\vv$.
Obtain $B_0$ from $B$ by replacing the columns corresponding to variable nodes $v_i\neq\vv$ with $i\in\cF(\vA^-)$ by all-zero columns.

Now assume that $\cE$ occurs and that for every $b\in\partial_{\vA^+}\vv\setminus\{a\}$ there exists $u\in\partial_{\vA^+}b\setminus\{\vv\}$ such that $\fm_{u\to b}(\vA^+)=\unfrozen$.
In fact, let $U=\{u\in\partial^2_{\vA^+}\vv:\fm_{u\to b}(\vA^+)=\unfrozen\}$.
Then $U\cap\cF(\vA^-)=\emptyset$, because $\cF(\vA^-)\subset\cF(\vA^+\setminus\{b\})$ for every $b\in\partial_{\vA^+}\vv$.
Due to {\bf E1} for every column representing a variable $u\in U$ the $u$-column of $B_0$ contains precisely one non-zero entry.
Therefore,  $\rank(B_0\ C)=|\partial_{\vA^*}\vv\setminus\{a\}|$, i.e., the matrix $(B_0\ C)$ has full row rank.
Since {\bf E2} ensures that $\partial^2_{\vA^+}\vv$ is not a proper relation of $\vA^-$, Fact~\ref{fact_indep} shows that
\begin{align}\label{eqclaim_mhatm2_1}
		\nul\begin{pmatrix} \vA^-\setminus\{\vv\}&0\\ B&C \end{pmatrix}=\nul(\vA^+\setminus\{a,\vv\})-|\partial_{\vA^*}\vv\setminus\{a\}|-1.
	\end{align}
Similarly, we can compute the rank of the matrix obtained by adding one more row with a single $1$-entry in the last column, thereby expressly pinning $\vv$:
\begin{align}\label{eqclaim_mhatm2_2}
	\nul\begin{pmatrix} \vA^-\setminus\{\vv\}&0\\ B&C\\0&1 \end{pmatrix}=\nul(\vA^+\setminus\{a,\vv\})-|\partial_{\vA^*}\vv\setminus\{a\}|-2<\nul\begin{pmatrix} \vA^-\setminus\{\vv\}&0\\ B&C \end{pmatrix}.
\end{align}
Combining \eqref{eqclaim_mhatm2_1}--\eqref{eqclaim_mhatm2_2}, we conclude that the last coordinate $n$ that represents $\vv$ is unfrozen in $\begin{pmatrix} \vA^-\setminus\{\vv\}&0\\ B&C \end{pmatrix}$; for otherwise the nullities on the left and right of \eqref{eqclaim_mhatm2_2} would have been equal.
Hence, \eqref{eqclaim_mhatm2_0} shows that $\vv$ is unfrozen in $\vA^+\setminus\{a\}$.
Thus, $\fm_{\vv\to a}(\vA^+)=\unfrozen$ by \eqref{eqWPdef}.
A similar argument yields the second assertion.
\end{proof}

We proceed to investigate the check-to-variable messages.

\begin{claim}\label{claim_mhatF1}
	Assume that $\cE$ occurs and let $a\in\partial_{\vA^+}\vv$.
	If $\fm_{w\to a}(\vA^+)=\frozen$ for all $w\in\partial_{\vA^+}\vv\setminus\{a\}$, then $\fm_{a\to \vv}(\vA^+)=\frozen$.
\end{claim}
\begin{proof}
	If $\fm_{w\to a}(\vA^+)=\frozen$, then $w\in\cF(\vA^+\setminus\{a\})$ by the definition \eqref{eqWPdef} of the standard messages.
	Hence, {\bf E3} guarantees that $w \in \cF(\vA^-)$ for all $w \in \partial_{\vA^+}\vv\setminus\{a\}$.
	Further, since $\cF(\vA^-) \subset \cF(\vA^-\setminus(\partial_{\vA^+}\vv\setminus\{a\}))$ we obtain from \eqref{eqWPdef} that $\fm_{a\to \vv}(\vA^+)=\frozen$. 
\end{proof}

\begin{claim}\label{claim_mhatF2}
	Assume that $\cE$ occurs and let $a\in\partial_{\vA^+}\vv$.
	If there exists $w\in\partial_{\vA^+}a\setminus\{\vv\}$ such that $\fm_{w\to a}(\vA^+)=\unfrozen$, then $\fm_{a\to \vv}(\vA^+)=\unfrozen$.
\end{claim}
\begin{proof}
	Let $w\in\partial_{\vA^+}a\setminus\{\vv\}$ be such that $\fm_{w\to a}(\vA^+)=\unfrozen$. 
	Then the definition \eqref{eqWPdef} of $\fm_{w\to a}(\vA^+)$ ensures that $w \notin \cF(\vA^+ \setminus\set{a})$.
	Since $\cF(\vA^-)\subset\cF(\vA^+\setminus\set{a})$, we conclude that $w \notin \cF(\vA^-)$.
	Further, for suitable permutation matrices $\Pi,\Pi'$ we obtain $D\in\FF_q^{n-1}$ and $\chi\in\FF_q\setminus\{0\}$ such that
	\begin{align}\label{eqclaim_mhatF2_0}
		\vA^+\setminus(\partial_{\vA^+}\vv\setminus\{a\})&=\Pi\cdot\begin{pmatrix} \vA^-&0\\ D&\chi \end{pmatrix}\cdot\Pi';
	\end{align}
	thus, the permutation matrices $\Pi,\Pi'$ are chosen such that they swap the $\vv$-column to the last column and the $a$-row to the last row.
	Hence, the last row $(D, \chi)$ represents $a$. 
	Now obtain $D_0$ from $D$ by replacing all entries corresponding to variable nodes from $\cF(\vA^+)\setminus\{\vv\}$ by $0$. 
	Then due to {\bf E2}, Fact~\ref{fact_indep} shows that
	\begin{align*}
		\nul\begin{pmatrix} \vA^-&0\\ D&\chi \end{pmatrix} &= \nul(\vA^-)&\mbox{and}&&
		\nul\begin{pmatrix} \vA^-&0\\ D&\chi \\ 0&1 \end{pmatrix} = 	\nul\begin{pmatrix} \vA^-&0\\ D&\chi \end{pmatrix} -1.
	\end{align*}
	Hence, as in the proof of Claim~\ref{claim_mhatm2} we obtain $\vv \notin \cF(\vA^+\setminus(\partial_{\vA^+}\vv\setminus\{a\}))$.
	Thus,  $\fm_{a \to \vv}(\vA^+) = \unfrozen$ by \eqref{eqWPdef}.
\end{proof}

\begin{proof}[Proof of \Prop~\ref{prop_vhA}]
Claims~\ref{claim_E}--\ref{claim_mhatF2} directly imply that
\begin{align*}
	\sum_{i=1}^m\sum_{j=1}^n\vecone\cbc{\fm_{v_j\to a_i}(\vhA)\neq\hat\fm_{v_j\to a_i}(\vhA)}&=o(n)&&\mbox{and}&
	\sum_{i=1}^m\sum_{j=1}^n\vecone\cbc{\fm_{a_i\to v_j}(\vhA)\neq\hat\fm_{a_i\to v_j}(\vhA)}&=o(n),
	\end{align*}
	whence we obtain \eqref{eqprop_vhAi}.
Similarly, \eqref{eqprop_vhAii} follows from Claims~\ref{claim_E}--\ref{claim_mhatm2}.

	Finally, in light of \eqref{eqprop_vhAii}, to prove \eqref{eqprop_vhAiii} it suffices to consider variables $v_j$ with $j\not\in\cF(\vhA)$.
	Hence, let $j,j'\in[n]\setminus\cF(\vhA)$ be two distinct indices such that $\{i,j\}$ is not a proper relation of $\vhA$; \Cor~\ref{cor_vhA} shows that this last property is violated for at most $o(n^2)$ pairs $j,j'$.
	Then the projection $\sigma\in\ker\vhA\mapsto(\sigma_j,\sigma_{j'})\in\FF_q^2$ is an epimorphism.
	Therefore, for any $s,t\in\FF_q^2$ we have $\abs{\{\sigma\in\ker\vhA:\sigma_i=s,\,\sigma_j=t\}}=q^{-2}|\ker\vhA|$.
	Consequently, if $\vhs\in\ker\vhA$ is drawn randomly, then for $(1-\vha+o(1))^2n^2$ pairs $j,j'\not\in\cF(\vhA)$ the random variables $\vhs_j,\vhs_{j'}$ are independent and uniformly distributed.
	Thus, Chebyshev's inequality shows that given $\cE$ for all $s\in\FF_q$ we have $$|\{j\in[n]\setminus\cF(\vhA):\vhs_j=s\}|=(1-\vha+o(1))|)n/q\qquad\mbox{ \whp,}$$ whence we obtain \eqref{eqprop_vhAiii}.
\end{proof}

\subsection{Proof of \Prop~\ref{prop_wp}}\label{sec_prop_wp}

The proof employs arguments broadly similar to those from the proof of \Prop~\ref{prop_vhA}.
The main difference is that we are going to consider a uniformly random pair $(\vv,\vv')$ of variable nodes, rather than a single variable node.
We begin by estimating the sizes $|\Delta_{z,\ell}(\fm_{\nix\to\nix}(\vhA))\times\Delta_{z',\ell'}(\fm_{\nix\to\nix}(\vhA))|$ for $z,z'\in\{\unfrozen,\slush,\frozen\}$, $\ell\in\cD(z)$ and $\ell'\in\cD(z')$.
Similarly as in \Sec~\ref{sec_prop_vhA} 
obtain $\vA^-$ from $\vhA$ by deleting all checks $a\in\partial_{\vhA}\vv\cup\partial_{\vhA}\vv'$.

\begin{fact}\label{fact_astar2}
	Let $\vA^+$ be the matrix obtained from $\vhA$ via the following operation.
	\begin{align*}
		\parbox{14cm}{Independently for all $a\in\partial_{\vhA}\vv\cup\partial_{\vhA}\vv'$ resample the neighbours of $a$ other than $\vv,\vv'$ uniformly without replacement from $\{v_1,\ldots,v_n\}\setminus\{\vv\}$.}
	\end{align*}
	Then $\vhA$ and $\vA^+$ have total variation distance $o(1)$.
\end{fact}
\begin{proof}
	Given that $\vv,\vv'$ have distance at least four in both $G(\vhA)$ and $G(\vA^+)$, the Tanner graphs of $\vhA$, $\vA^+$ are identically distributed.
	Moreover, the probability that $\vv,\vv'$ have distance less than four is bounded by $n^{-1+o(1)}$.
\end{proof}

The plan is to derive the following joint probability formula, and then follow up with Chebyshev's inequality.

\begin{lemma}\label{lem_Dzl}
	\Whp\ we have $\pr\brk{\vv\in\Delta_{z,\ell}(\fm_{\nix\to\nix}(\vA^+)),\vv'\in\Delta_{z',\ell'}(\fm_{\nix\to\nix}(\vA^+))\mid\vhA}=\bar\Delta_{z,\ell}(\vha)\bar\Delta_{z',\ell'}(\vha)+o(1)$.
\end{lemma}

Towards the proof of \Lem~\ref{lem_Dzl} let $\cE'$ be the event that the following statements hold; let $\Delta\gg1$ diverge sufficiently slowly.
\begin{description}
	\item[E0$'$] we have $|\partial_{\vA^+}\vv|=\ell_{\unfrozen\unfrozen}+\ell_{\frozen\unfrozen}+\ell_{\unfrozen\frozen}+\ell_{\frozen\frozen}$ and $|\partial_{\vA^+}\vv'|=\ell'_{\unfrozen\unfrozen}+\ell'_{\frozen\unfrozen}+\ell'_{\unfrozen\frozen}+\ell'_{\frozen\frozen}$.
	\item[E1$'$] the second neighbourhoods $\partial_{\vA^+}^2\vv,\partial_{\vA^+}^2\vv'$ satisfy
		\begin{align*}
			\vv,\vv'&\not\in\partial_{\vA^+}^2\vv\cup\partial_{\vA^+}^2\vv',&
			|\partial_{\vA^+}^2\vv|&=(k-1)|\partial_{\vA^+}\vv|\leq\Delta,&
			|\partial_{\vA^+}^2\vv'|&=(k-1)|\partial_{\vA^+}\vv'|\leq\Delta.
		\end{align*}
	\item[E2$'$] we have $|\cF(\vhA)\setminus\cF(\vA^-)|=o(n)$ and $|\cF(\vA^+)\setminus\cF(\vA^-)|=o(n)$.
	\item[E3$'$] for all $a\in\partial_{\vA^+}\vv\cup\partial_{\vA^+}\vv'$ we have
			$\cF(\vA^+\setminus\{a\})\cap\partial_{\vA^+}a\setminus\{\vv,\vv'\}=\cF(\vA^-)\cap\partial_{\vA^+}a\setminus\{\vv,\vv'\}$.
	\item[E4$'$] for all $v\in\{\vv,\vv'\}$ and $a\in\partial_{\vA^+}v$ we have $\fm_{v\to a}(\vA^+)=\hat\fm_{v\to a}(\vA^+)$, $\fm_{a\to v}(\vA^+)=\hat\fm_{a\to v}(\vA^+)$.
\end{description}
Thus, {\bf E0$'$} provides that the degrees of $\vv,\vv'$ match the sum of the entries of $\ell,\ell'$.
Moreover, {\bf E1$'$} ensures that $\vv,\vv'$ have distance at least four and that their second neighbourhoods are acyclic.
Further, {\bf E2$'$} provides that $\vA^-,\vhA,\vA^+$ have about the same number of frozen variables.
In particular, {\bf E3$'$} demands that the frozen variables in the second neighbourhood of $\vv,\vv'$ coincide in $\vA^+$ and $\vA^-$.
Finally, {\bf E4} posits that the messages that touch $\vv,\vv'$ are invariant under the WP update \eqref{eqWPupdate1}.

\begin{claim}\label{claim_Dzl}
	We have $\pr\brk{\cE'\mid\mbox{\bf E0$'$}}=1-o(1)$ and
	\begin{align}\label{eqclaim_Dzl}
		\pr\brk{\mbox{\bf E0$'$}}=\pr\brk{\Po(d)=\ell_{\unfrozen\unfrozen}+\ell_{\frozen\unfrozen}+\ell_{\unfrozen\frozen}+\ell_{\frozen\frozen}}\pr\brk{\Po(d)=\ell'_{\unfrozen\unfrozen}+\ell'_{\frozen\unfrozen}+\ell'_{\unfrozen\frozen}+\ell'_{\frozen\frozen}}+o(1).
	\end{align}
\end{claim}
\begin{proof}
	The estimate \eqref{eqclaim_Dzl} is an immediate consequence of Fact~\ref{fact_degs}.
	Regarding the probability of $\cE'$ given {\bf E0$'$}, the same arguments as in the proof of Claim~\ref{claim_E} show that {\bf E1$'$}--{\bf E3$'$} follow from Fact~\ref{fact_degs}, \Cor~\ref{cor_vhA} and \Cor~\ref{cor_tinker}.
	Furthermore, {\bf E4$'$} follows from Eq.~\eqref{eqprop_vhAii} from \Prop~\ref{prop_vhA} and Fact~\ref{fact_astar2}.
\end{proof}

\begin{proof}[Proof of \Lem~\ref{lem_Dzl}]
	Let $\vX=|\partial_{\vA^+}\vv|$, $\vX'=|\partial_{\vA^+}\vv'|$ be the degrees of $\vv,\vv'$.
	Moreover, let
	\begin{align*}
		\vX_\frozen&=\sum_{a\in\partial_{\vA^+}\vv}\vecone\{\partial_{\vA^+}a\setminus\{\vv\}\subset\cF(\vA^-)\},&\vX_\unfrozen&=\vX-\vX_\frozen,&\vX'_\frozen&=\sum_{a\in\partial_{\vA^+}\vv'}\vecone\{\partial_{\vA^+}a\setminus\{\vv'\}\subset\cF(\vA^-)\},&\vX'_\unfrozen&=\vX'-\vX'_\frozen.
	\end{align*}
	Additionally, let $\cX=\{\vX_{\frozen}=\ell_{\frozen\frozen}+\ell_{\frozen\unfrozen},\,\vX_{\unfrozen}=\ell_{\unfrozen\frozen}+\ell_{\unfrozen\unfrozen}\}$
	and $\cX'=\{\vX'_{\frozen}=\ell'_{\frozen\frozen}+\ell'_{\frozen\unfrozen},\,\vX'_{\unfrozen}=\ell'_{\unfrozen\frozen}+\ell'_{\unfrozen\unfrozen}\}$.
	We are going to argue that \Prop~\ref{prop_vhA} and Claim~\ref{claim_Dzl} imply
	\begin{align}\label{eqlem_Dzl1}
		\pr\brk{\vv\in\Delta_{z,\ell}(\fm_{\nix\to\nix}(\vA^+)),\vv'\in\Delta_{z',\ell'}(\fm_{\nix\to\nix}(\vA^+))\mid\cE'}&=\pr\brk{\cX\cap\cX'\mid\cE'}+o(1).
	\end{align}
	Indeed, the WP fixed point property {\bf E4$'$} ensures that the WP messages that $\vv,\vv'$ send out to their neighbouring check nodes are determined by the incoming messages.
	Furthermore, {\bf E3$'$} provides that for every $a\in\partial_{\vA^+}\vv$ we have $\fm_{a\to\vv}(\vA^+)=\frozen$ iff $\partial_{\vA^+}a\setminus\{\vv\}\subset\cF(\vA^-)$, and similarly for $\vv'$.
	Consequently, \eqref{eqWPupdate1}, \eqref{eqWPmarks1} and~\eqref{eqWPmarks2} show that on $\cE'$ the random variables $\vX_{\frozen},\vX_{\unfrozen},\vX'_{\frozen},\vX'_{\unfrozen}$ capture the salient information supplied by the incoming messages $\fm_{\nix\to\vv}(\vA^+),\fm_{\nix\to\vv'}(\vA^+)$, whence we obtain \eqref{eqlem_Dzl1}.
	
	Further, we claim that if $\vhA$ satisfies {\bf E0$'$}, then
	\begin{align}\label{eqlem_Dzl2}
		\pr\brk{\cX\cap\cX'\mid\vhA}&=\vec\alpha^{(k-1)(\ell_{\frozen\unfrozen}+\ell_{\frozen\frozen}+\ell'_{\frozen\unfrozen}+\ell'_{\frozen\frozen})}
		(1-\vec\alpha^{k-1})^{\ell_{\unfrozen\unfrozen}+\ell_{\unfrozen\frozen}+\ell'_{\unfrozen\unfrozen}+\ell'_{\unfrozen\frozen}}+o(1);
	\end{align}
	for by construction the new second neighbours of $\vv,\vv'$ are chosen uniformly.
	Hence, due to {\bf E2$'$} the probability that any specific second neighbour belongs to $\cF(\vA^-)$ equals $\vec\alpha+o(1)$, and due to {\bf E1$'$} these events are asymptotically independent.
	Finally, we combine \eqref{eqclaim_Dzl}, \eqref{eqlem_Dzl1} and \eqref{eqlem_Dzl2} to complete the proof.
\end{proof}

In order to estimate the sizes of the sets $\Gamma_{z,\ell}(\fm_{\nix\to\nix}(\vhA))$, we let $\va,\va'$ be a random pair of distinct check nodes.
Let $\vA^{\#}$ be the matrix obtained from $\vhA$ by resampling the neighbours of $\va,\va'$ independently.
Then $\vA^{\#}$ and $\vhA$ are identically distributed.
In analogy to \Lem~\ref{lem_Dzl}, we prove the following.

\begin{lemma}\label{lem_Gzl}
	Let $z,z'\in\{\unfrozen,\frozen,\slush\}$ and let $\ell\in\cG(z),\ell'\in\cG(z')$.
	\Whp\ we have $$\pr\brk{\va\in\Gamma_{z,\ell}(\fm_{\nix\to\nix}(\vA^\#)),\va'\in\Gamma_{z',\ell'}(\fm_{\nix\to\nix}(\vA^\#))\mid\vhA}=\bar\Gamma_{z,\ell}(\vha)\bar\Gamma_{z',\ell'}(\vha)+o(1).$$
\end{lemma}
\begin{proof}
	Consider the following event $\cA$:
	\begin{description}
	\item[A1] the neighbourhoods $\partial_{\vA^\#}\va,\partial_{\vA^\#}\va'$ are disjoint.
	\item[A2] we have $|\cF(\vA^\#)\setminus\cF(\vhA\setminus\{\va,\va'\})|=o(n)$ and $|\cF(\vhA)\setminus\cF(\vhA\setminus\{\va,\va'\})|=o(n)$.
	\item[A3] we have $\cF(\vA^\#\setminus\{\va\})\cap\partial_{\vA^\#}\va=\cF(\vhA\setminus\{\va,\va'\})\cap\partial_{\vA^\#}\va$ and $\cF(\vA^\#\setminus\{\va'\})\cap\partial_{\vA^\#}\va'=\cF(\vhA\setminus\{\va,\va'\})\cap\partial_{\vA^\#}\va'$.
	\item[A4] for all $v\in\partial_{\vA^\#}\va$ we have $\fm_{\va\to v}(\vA^\#)=\hat\fm_{\va\to v}(\vA^\#)$ and for all $v\in\partial_{\vA^\#}\va'$ we have $\fm_{\va'\to v}(\vA^\#)=\hat\fm_{\va'\to v}(\vA^\#)$.
	\end{description}
	Then \Cor~\ref{cor_vhA}, \Prop~\ref{prop_vhA}  and \Cor~\ref{cor_tinker} show that
	\begin{align}\label{eqlem_Gzl0}
		\pr\brk\cA&=1-o(1).
	\end{align}
	
	Further, let
	\begin{align*}
		\vY_\frozen&=|\partial_{\vA^\#}\va\cap\cF(\vhA\setminus\{\va,\va'\})|,&
		\vY_\unfrozen&=k-\vY_\frozen,&
		\vY'_\frozen&=|\partial_{\vA^\#}\va'\cap\cF(\vhA\setminus\{\va,\va'\})|,&
		\vY'_\unfrozen&=k-\vY_\frozen'.
	\end{align*}
	Also let $\cY=\{\vY_\frozen=\ell_{\frozen\frozen}+\ell_{\frozen\unfrozen}\}$ and $\cY'=\{\vY'_\frozen=\ell'_{\frozen\frozen}+\ell'_{\frozen\unfrozen}\}$.
	We claim that 
	\begin{align}\label{eqlem_Gzl1}
		\pr\brk{\va\in\Gamma_{z,\ell}(\fm_{\nix\to\nix}(\vA^\#)),\vv'\in\Gamma_{z',\ell'}(\fm_{\nix\to\nix}(\vA^\#))\mid\cA}&=\pr\brk{\cY\cap\cY'\mid\cA}+o(1);
	\end{align}
	for {\bf A4} provides that the messages that $\va,\va'$ send out to their neighbours are determined by the incoming messages via \eqref{eqWPupdate1}.
	Moreover, {\bf A3} ensures that for $v\in\partial_{\vA^\#}\va$ we have $\fm_{v\to\va}(\vA^\#)=\frozen$ iff $v\in\cF(\vhA\setminus\{\va,\va')$, and similarly for $v'\in\partial_{\vA^\#}\va'$.

	Finally, since $\vA^\#$ is obtained by resampling the neighbourhoods of $\va,\va'$, {\bf A1}--{\bf A2} show that
	\begin{align}\label{eqlem_Gzl2}
		\pr\brk{\cY\cap\cY'\mid\vhA}&=\vec\alpha^{\ell_{\frozen\unfrozen}+\ell_{\frozen\frozen}+\ell'_{\frozen\unfrozen}+\ell'_{\frozen\frozen}}(1-\vec\alpha)^{\ell_{\unfrozen\unfrozen}+\ell_{\unfrozen\frozen}+\ell'_{\unfrozen\unfrozen}+\ell'_{\unfrozen\frozen}}+o(1).
	\end{align}
	Thus, the assertion follows from \eqref{eqlem_Gzl0}--\eqref{eqlem_Gzl2}.
\end{proof}

\begin{proof}[Proof of \Prop~\ref{prop_wp}]
	The proposition follows from Fact~\ref{fact_degs}, \Lem s~\ref{lem_Dzl} and~\ref{lem_Gzl} and Chebyshev.
\end{proof}

\section{Moment computations}\label{sec_covers}

\noindent
In this section we prove \Prop~\ref{prop_annealedwp} and \Cor~\ref{cor_annealedwp} and complete the proof of \Thm~\ref{thm_main}.
Our principal tool will be moment computations.
In particular, we will compute the mean of the number $\vX_\alpha$ of $\alpha$-extensions for $\alpha\in[0,1]$.
Crucially, because the definitions \eqref{eqalphaWP1}--\eqref{eqalphaWP2} prescribe the correct `quenched' statistics provided by \eqref{eqDeltafell}--\eqref{eqGammafell} as well as an approximate version of the WP fixed point property \eqref{eqWPfixedEx}, the ensuing calculations turn out to be tight as well as relatively elegant.
This manifests itself in the fact that we ultimately recover the function $\Phi_{d,k}$ from~\eqref{eqPhi}.

\subsection{Counting WP fixed points}\label{sec_WP_count}
We begin by calculating the expected number of $\alpha$-WP fixed points, for which we resort to the pairing model of the random bipartite Tanner graph.
To this end we condition on the $\sigma$-algebra $\fD$ generated by the degrees $d_{\vhA}(v_j)$ of the variable nodes and by $\vt$.
Given $\fD$ let
\begin{align*}
	\fV=\bigcup_{j=1}^n\{v_j\}\times[d_{\vhA}(v_j)]&&\mbox{and}&&\fF=\cbc{b_1\cup\cdots\cup b_{\vt}}\cup\bigcup_{i=1}^m\cbc{a_i}\times[k]
\end{align*}
be sets of variable and check clones; here $b_1,\ldots,b_{\vt}$ represent the checks that the pinning operation from \Sec~\ref{sec_pin} induces.
A {\em pairing} is a bijection $\pi:\fV\to\fF$. 
Let $\fP$ be the set of all pairings.
As usual, we construct a Tanner graph $G(\vec\pi)$ by drawing a $\vec\pi\in\fP$ uniformly at random and contracting the clones into single vertices.
This graph may possess multi-edges, in contrast to the random graph $G(\vhA)$.
However, it is well known that once we condition on the event $\fS$ that $G(\vec\pi)$ is simple, the distribution of $G(\vec\pi)$ coincides with that of $G(\vhA)$.
Moreover, routine arguments along the lines of~\cite[Chapter~9]{JLR} show the following.

\begin{fact}
	For any $d>0,k\geq3$ \whp\ we have $\pr\brk{\fS\mid\fD}=\Omega(1)$.
\end{fact}

In order to calculate the expected number of $\alpha$-WP fixed points of $G(\vec\pi)$ we compute the total number of pairings $\pi\in\fP$ together with appropriate $\{\unfrozen,\frozen\}$-valued annotations of the clones.
To be precise, an {\em $\alpha$-cover} $(\pi,\fp)$ consists of a pairing $\pi$ and a map $\fp:\fV\cup\fF\to\{\unfrozen,\frozen\}^2,(x,h)\mapsto\fp(x,h)=(\fp_1(x,h),\fp_2(x,h))$ that satisfy the following conditions.
\begin{description}
	\item[COV1] For all $(x,h)\in\fV\cup\fF$ we have $(\fm_1(\pi(x,h)),\fm_2(\pi(x,h)))=(\fm_2(x,h),\fm_1(x,h))$.
	\item[COV2] For all but $o(n)$ pairs $(v_j,l)$ with $j\in[n]$ and $l\in[d_{\vhA}(v_i)]$ we have
		\begin{align*}
			\fp_2(v_j,l)=\begin{cases}
				\frozen&\mbox{ if $\fp_1(v_j,h)=\frozen$ for some $h\in[d_{\vhA}(v_j)]\setminus\{l\}$},\\
				\unfrozen&\mbox{ otherwise.}
			\end{cases}
		\end{align*}
	\item[COV3] For all but $o(n)$ pairs $(a_i,l)$ with $i\in[m]$ and $l\in[d_{\vhA}(a_i)]$ we have
		\begin{align*}
			\fp_2\bc{a_i,l}=\begin{cases}
				\frozen&\mbox{ if $\fp_1\bc{a_i,h}=\frozen$ for all $h\in[k]\setminus\{l\}$},\\
				\unfrozen&\mbox{ otherwise.}
			\end{cases}
		\end{align*}
	\item[COV4] For any $z\in\{\frozen,\slush,\unfrozen\}$, $\ell = (\ell_{\unfrozen\unfrozen},\ell_{\unfrozen\frozen},\ell_{\frozen\unfrozen},\ell_{\frozen\frozen})\in\cL$, $i\in[m]$ and $j\in[n]$ let
		\begin{align}\label{eqCOV5a}
			\fp(v_j)&=
			\begin{cases}
				\frozen&\mbox{ if $\fp_{1}(v_j,l)=\frozen$ for at least two $l\in[d_{\vhA}(v_j)]$,}\\
				\slush&\mbox{ if $\fp_{1}(v_j,l)=\frozen$ for precisely one $l\in[d_{\vhA}(v_j)]$,}\\
				\unfrozen&\mbox{ otherwise,}
			\end{cases}\quad\\
			\fp(a_i)&=\begin{cases}
				\frozen&\mbox{ if $\fp_{1}(a_i,l)=\frozen$ for all $l\in[d_{\vhA}(a_i)]$,}\\
				\slush&\mbox{ if $\fp_{1}(a_i,l)=\frozen$ for all but precisely one $l\in[d_{\vhA}(a_i)]$,}\\
				\unfrozen&\mbox{ otherwise,}
			\end{cases}\qquad
			\label{eqCOV5b}\\
				\vDelta(z,\ell)&=\sum_{i=1}^n
					\vecone\cbc{\fp\bc{v_j}=z}\prod_{x,y\in\{\unfrozen,\frozen\}}\vecone\cbc{\abs{\cbc{l\in[d_{\vhA}(v_j)]:\fp_1(v_j,l)=x, \; \fp_2(v_j,l)=y}}=\ell_{x y}},			\label{eqCOV5c}\\
				\vGamma( z,\ell)&=
				\sum_{i=1}^m\vecone\cbc{\fp(a_i)=z}\prod_{x,y \in\{\unfrozen,\frozen\}}\vecone\cbc{\abs{\cbc{l\in[d_{\vhA}(a_i)]:\fp_1(a_i,l)=x,\;\fp_2(a_i,l)=y}}=\ell_{x y}}.			\label{eqCOV5d}
			\end{align}
			Then 
			\begin{align}\label{eqCOV5punch}
				\vDelta(z,\ell)&=n\bDelta_{z,\ell}(\alpha)+o(n),&\vGamma(z,\ell)&=m\bGamma_{z,\ell}(\alpha)+o(n).
			\end{align}
	\end{description}
Condition {\bf COV1} provides consistency of the labels associated with the paired clones.
Moreover, {\bf COV2}--{\bf COV3} impose the fixed point condition \eqref{eqWPfixedEx} on $(\pi,\fp)$.
Similarly, the labels \eqref{eqCOV5a}--\eqref{eqCOV5b} mimic the definitions \eqref{eqWPmarks1}--\eqref{eqWPmarks2}.
Finally, \eqref{eqCOV5c}--\eqref{eqCOV5punch} ensure that the statistics of the labels/messages are in line with the correct `quenched' values~\eqref{eqDeltafell}--\eqref{eqGammafell} (see \Prop~\ref{prop_wp}).
The following lemma determines the size of the set $\fC(\alpha)$ of all $\alpha$-covers.

\begin{lemma}\label{lemma_cover}
	\Whp\ we have $\fC(\alpha)=\exp(o(n))(km)!k!^m\prod_{i=1}^nd_{\vhA}(v_i)!\enspace$.
\end{lemma}

To prove \Lem~\ref{lemma_cover} we begin with the following straightforward counting formula.

\begin{claim}\label{claim_cover1}
	With $y,y'$ ranging over $\{\unfrozen,\frozen\}$, $z$ ranging over $\{\unfrozen,\slush,\frozen\}$ and $\ell$ ranging over $\cL$ we have \whp
	\begin{align}\nonumber
		\frac{|\fC(\alpha)|}{(km)!}&=\exp(-nH(\Po(d))+o(n))\binom{n}{n(\bar\Delta_{z,\ell}(\alpha))_{z,\ell}} \binom{m}{m(\bar\Gamma_{z,\ell}(\alpha))_{z,\ell}}\\
								   &\qquad\qquad\qquad\qquad\cdot\binom{km}{\bc{n\sum_{z,\ell}\ell_{yy'}\bar\Delta_{z,\ell}(\alpha)}_{y,y'}}^{-1}\prod_{z,\ell}\binom{\ell_{\unfrozen \unfrozen}+\ell_{\unfrozen \frozen }+\ell_{\frozen \unfrozen }+\ell_{\frozen \frozen }}{\ell_{\unfrozen \unfrozen },\ell_{\unfrozen \frozen },\ell_{\frozen \unfrozen },\ell_{\frozen \frozen }}^{n\bar\Delta_{z,\ell}(\alpha)+m\bar\Gamma_{z,\ell}(\alpha)}.
\label{eqcover1}
	\end{align}
\end{claim}
\begin{proof}
	The first two multinomial coefficients account for the number of ways of assigning labels with the frequencies prescribed by \eqref{eqCOV5punch} to the variables/checks.
	However, the first multinomial coefficient implicitly counts the assignment of the variable node degrees, on which we condition; this is because $\ell_{\unfrozen\unfrozen}+\ell_{\frozen\unfrozen}+\ell_{\unfrozen\frozen}+\ell_{\frozen\frozen}$ equals the degree of the corresponding variable.
	To correct for this overcounting, we divide by the multinomial coefficient
	\begin{align}\label{eqmultinomial}
		\binom{n}{(|\{i\in[n]:d_{\vhA}(v_i)=h\}|)_{h\geq0}}.
	\end{align}
	But since the variable node degrees are asymptotically Poisson by Fact~\ref{fact_degs}, \eqref{eqmultinomial} equals $\exp(nH(\Po(d))+o(n))$ \whp\
	The first multinomial coefficient on the second line of \eqref{eqcover1} counts the number of possible matchings of the clones that respect {\bf COV1}.
	The last factor accounts for the number of ways of assigning labels to the clones of the individual variable/check nodes.
	Finally, the $\exp(o(n))$ error term swallows the approximations in \eqref{eqCOV5c}--\eqref{eqCOV5d}.
\end{proof}

\begin{claim}\label{claim_cover2}
	Letting 
	\begin{align*}
		\fl_1&=\Erw[\log(\Po(d)!)],\qquad\fl_2=-\sum_{z,\ell}\bar\Delta_{z,\ell}(\alpha)\log(\ell_{\unfrozen\unfrozen}!\ell_{\unfrozen\frozen}!\ell_{\frozen\unfrozen}!\ell_{\frozen\frozen}!),\qquad\fl_3=-\frac dk\sum_{z,\ell}\bar\Gamma_{z,\ell}(\alpha)\log(\ell_{\unfrozen\unfrozen}!\ell_{\unfrozen\frozen}!\ell_{\frozen\unfrozen}!\ell_{\frozen\frozen}!),\\
		\fh_1&=H(\bar\delta(\alpha,z))_z+H(\Po(d(1-\alpha^{k-1})))+\bar\delta(\alpha,f)H(\Po_{\geq2}(d\alpha^{k-1})),\\
		\fh_2&=\frac dk\brk{ H(\bar\gamma(\alpha,z))_z+\bar\gamma(\alpha,\unfrozen)H(\Bin_{\geq2}(k,1-\alpha))},\quad \fh_3=d\brk{H(\Be(\alpha^{k-1}))-H(\Be(\alpha))},\quad \fh_4= -H(\Po(d))
	\end{align*}
	\whp\ we have $\displaystyle\frac1n\log\frac{|\fC(\alpha)|}{(k!)^m(km)! \prod_{i=1}^nd_{\vhA}(v_i)! }=\fl_1+\fl_2+\fl_3+\fh_1+\fh_2+\fh_3 +\fh_4+ o(1).$ 
\end{claim}
\begin{proof}
	In combination with \eqref{eqdeltagamma1}--\eqref{eqGammafell}, Stirling's formula shows that
	\begin{align}\label{eq_claim_cover2_1}
	\frac1n\log\binom{n}{n(\bar\Delta_{z,\ell}(\alpha))_{z,\ell}}&=H(\bar\delta(\alpha,z))_z+H(\Po(d(1-\alpha^{k-1})))+\bar\delta(\alpha,f)H(\Po_{\geq2}(d\alpha^{k-1}))+o(1),\\
\frac1n\log\binom{m}{m(\bar\Gamma_{z,\ell}(\alpha))_{z,\ell}}&=\frac dk\bc{H(\bar\gamma(\alpha,z))_z+\bar\gamma(\alpha,\unfrozen)H(\Bin_{\geq2}(k,1-\alpha))}+o(1).\label{eq_claim_cover2_2}
	\end{align}
	Similarly,
	\begin{align}\label{eq_claim_cover2_3}
		\frac1n\log\brk{\frac{1}{(km)!}\prod_{y,y'}\bc{n\sum_{z,\ell}\ell_{y,y'}\bar\Delta_{z,\ell}(\alpha)}!}&=-d\brk{H(\Be(\alpha))-H(\Be(\alpha^{k-1}))}+o(1).
	\end{align}
	Further,
	\begin{align}\label{eq_claim_cover2_4}
		\sum_{z,\ell}\bar\Delta_{z,\ell}(\alpha)\log\frac{(\ell_{\unfrozen \unfrozen}+\ell_{\unfrozen \frozen }+\ell_{\frozen \unfrozen }+\ell_{\frozen \frozen })!}{\ell_{\unfrozen \unfrozen }!\ell_{\unfrozen \frozen }!\ell_{\frozen \unfrozen }!\ell_{\frozen \frozen }!}
			&=\fl_1+\fl_2+o(1).
	\end{align}
	Finally, since $\ell_{\unfrozen \unfrozen}+\ell_{\unfrozen \frozen }+\ell_{\frozen \unfrozen }+\ell_{\frozen \frozen }=k$ for all $\ell$ such that $\bar\Gamma_{z,\ell}(\alpha)>0$, we have
\begin{align}\label{eq_claim_cover2_5}
		- \frac{1}{n}\log({(k!)^m}) + \frac m n \sum_{z,\ell}\bar\Gamma_{z,\ell}(\alpha)\log\frac{(\ell_{\unfrozen \unfrozen}+\ell_{\unfrozen \frozen }+\ell_{\frozen \unfrozen }+\ell_{\frozen \frozen })!}{\ell_{\unfrozen \unfrozen }!\ell_{\unfrozen \frozen }!\ell_{\frozen \unfrozen }!\ell_{\frozen \frozen }!} 
			&=\fl_3+o(1).
	\end{align}
	Combining \eqref{eq_claim_cover2_1}--\eqref{eq_claim_cover2_5} with Claim~\ref{claim_cover1} completes the proof.
\end{proof}

\begin{proof}[Proof of \Lem~\ref{lemma_cover}]
	Let $\lambda=\alpha^{k-1}d$ and $\mu=d-\lambda$.
	Since by Fact~\ref{fact_degs} the empirical distribution of the degrees $(d_{\vhA}(v_i))_{i\in[n]}$ is approximately $\Po(d)$ and in light of \eqref{eqdeltagamma1}--\eqref{eqGammafell}, \whp\ we have
	\begin{align}\label{eqlemma_cover1}
		\fl_1&=\frac1n\sum_{i=1}^n\log(d_{\vhA}(v_i)!)+o(1),\quad \fl_2=-\ex\brk{\log\bc{\Po(\mu)!}}-\bar\delta(\alpha,\frozen)\ex\brk{\log(\Po_{\geq2}(\lambda)!)}+o(1),
		\\
		\fl_3&=\frac dk\bar\gamma(\alpha,\slush)\log(k)+\frac dk\bar\gamma(\alpha,\unfrozen)\ex\brk{\log\binom{k}{\Bin_{\geq2}(k,1-\alpha)}}
		+o(1).\label{eqlemma_cover3}
	\end{align}
	Furthermore, trite rearrangements reveal that
	\begin{align}\label{eqlemma_cover4}
		\fh_1&=d(1-H(\Be(\alpha^{k-1}))-\log d)+\ex\brk{\log(\Po(\mu)!)}+\bar\delta(\alpha,\frozen)\ex\brk{\log(\Po_{\geq2}(\lambda)!)},\\
		\fh_2&=dH(\Be(\alpha))-\frac dk\bar\gamma(\alpha,\slush)\log(k)-\frac dk\bar\gamma(\alpha,\unfrozen)\ex\brk{\log\binom{k}{\Bin_{\geq2}(k,1-\alpha)}},\\
		\fh_4 &= -d(1-\log d) - \ex\brk{\log(\Po(d)!)}.
				 \label{eqlemma_cover5}
	\end{align}
	The assertion follows from \eqref{eqlemma_cover1}--\eqref{eqlemma_cover5} and Claim~\ref{claim_cover2}.
\end{proof}

\subsection{Proof of \Prop~\ref{prop_annealedwp}}\label{sec_prop_annealedwp}

\Lem~\ref{lemma_cover} estimates of the number of $\alpha$-WP fixed points.
In order to prove \Prop~\ref{prop_annealedwp} we now need to count the number of `balanced' assignments of values to the unfrozen variables of a WP fixed point such that all checks are satisfied.
Thus, let $(\pi,\fp)$ be an $\alpha$-cover.
Call $\sigma\in\FF_q^n$ {\em compatible} with $(\pi,\fp)$ if 
\begin{align} \label{eqprop_annealedwp}
	\sigma_j&=0\mbox{ for all $j\in[n]$ with $\fp(v_j)\neq\unfrozen$, and}\\
	\sum_{\ell\geq0}&\sum_{s\in\FF_q\setminus\cbc0}(\ell+1)\abs{\sum_{j=1}^n \vecone\{d_{\vhA}(v_j)=\ell,\,\fp(v_j)=\unfrozen\}\bc{\vecone\{\sigma_j=s\}-q^{-1}}}=o(n).\label{eqprop_annealedwp0}
\end{align}
Thus, we ask that the values of the variables $v_j$ with $\fp(v_j)=\unfrozen$ be about uniformly distributed on $\FF_q$, even when broken down to individual variable degrees.
Further, a pairing $\pi\in\fP$ induces a matrix $A(\pi)$ by letting
	\begin{align*}
		A_{ij}(\pi)&=\fA_{ij}\cdot\vecone\{\exists l\in[d_{\vhA}(v_j),\,h\in[k]:\pi(v_j,l)=(a_i,h)\}&&(i\in[m],\,j\in[n]).
	\end{align*}
Finally, we say that $\sigma\in\FF_q^n$ {\em essentially satisfies} $(\pi,\fp)$ if $\|A(\pi)\sigma\|_0=o(n)$. 
Recall that $\vec\pi\in\fP$ denotes a random pairing.

\begin{lemma}\label{lem_extension}
	Let $\fp:\fV\cup\fF\to\{\unfrozen,\frozen\}$ and let $\sigma\in\FF_q^n$. Let $\fC$ be the event that $(\vec\pi,\fp)$ is an $\alpha$-cover that $\sigma$ is compatible with, and let $\fE$ be the event that $\sigma$ is essentially satisfying.
	Then $\pr\brk{\fE\mid\fC,\fD}\leq q^{-m\bar\gamma(\alpha,\unfrozen)+o(n)}$ \whp\ 
\end{lemma}
\begin{proof}
	Given $\fC,\fD$ let $\fI$ be the set of all pairs $(i,h)\in[m]\times[k]$ such that $\vec\pi(i,h)\in\{v_j\}\times\NN$ for some variable $v_j$ with $\fp(v_j)=\unfrozen$.
	Thus, $\fI$ contains the check clones `hit' by an unfrozen variable.
	Further, let $\cI$ contain all $i\in[m]$ such that $\{i\}\times[k]\cap\fI\neq\emptyset$.
	What remains random given $\fC,\fD,\fI$ is {\em which} unfrozen variable clones are matched to $\fI$.
	Our goal is to estimate the probability that all checks $a_i$, $i\in\cI$, end up satisfied under this random matching.
	Let $\vec\xi=(\vec\xi_{ih})_{(i,h)\in\fI}$ be the vector that comprises the values under $\sigma$ of the variables that the clones in $\fI$ get matched to.
	In symbols, $\vec\xi_{ih}=\sum_{j\in[n]}\sigma_j\vecone\{\vec\pi(a_i,h)\in\{v_j\}\times\NN\}$.

	To investigate $\vec\xi$ we introduce an auxiliary random vector $\vec\chi=(\vec\chi_{ih})_{(i,h)\in\fI}$ with independent uniformly distributed entries $\vec\chi_{ih}\in\field_q$.
	Consider the events
	\begin{align*}
		\fR&=\cbc{\forall s\in\FF_q\setminus\{0\}:\sum_{(i,h)\in\fI}\vecone\{\vec\chi_{ih}=s\}=\sum_{j=1}^n\vecone\{\sigma_j=s\}d_{\vhA}(v_j)},&
		\fX&=\cbc{\sum_{i\in\cI}\vecone\cbc{\sum_{h:(i,h)\in\fI}\vec\chi_{ih}\neq0}=o(n)}.
	\end{align*}
	Given the event $\fR$ the vectors $\vec\xi$ and $\vec\chi$ are identically distributed.
	Hence,
	\begin{align}\label{eqlem_extension0}
		\pr\brk{\fE\mid\fC,\fD}&=\pr\brk{\fX\mid\fC,\fD,\fI,\fR}.
	\end{align}

	The unconditional probabilities $\pr\brk{\fX\mid\fC,\fD,\fI}$ and $\pr\brk{\fR\mid\fC,\fD,\fI}$ are computed easily.
	Indeed, because the $\vec\chi_{ih}$ are uniform and independent, for any $i\in\cI$ the event $\sum_{h:(i,h)\in\fI}\vec\chi_{ih}=0$ occurs with probability $1/q$.
	Hence,
	\begin{align}\label{eqlem_extension1}
		\pr\brk{\fS\mid\fC,\fD,\fI}&=q^{-|\cI|+o(n)}.
	\end{align}
	Furthermore, conditions {\bf COV1--COV4} and the definitions \eqref{eqGammafell0}--\eqref{eqGammafell} of the coefficients $\Gamma_{z,\ell}(\alpha)$ ensure that \whp\ given $\fC,\fD$ we have $|\fI|=m(\bar\gamma(\alpha,\unfrozen)+o(1))$.
	Thus, \eqref{eqlem_extension1} becomes
	\begin{align}\label{eqlem_extension2}
		\pr\brk{\fS\mid\fC,\fD,\fI}&=q^{-m\bar\gamma(\alpha,\unfrozen)+o(n)}.
	\end{align}
	Moreover, \eqref{eqprop_annealedwp0} ensures that $\pr\brk{\fR\mid\fC,\fD,\fI}=\exp(o(n))$.
	Combining \eqref{eqlem_extension0} and \eqref{eqlem_extension2} with Bayes' rule, we obtain
	\begin{align*}
		\pr\brk{\fE\mid\fC,\fD}=\ex\brk{\pr\brk{\fS\mid\fC,\fD,\fI,\fR}\mid\fC,\fD}
			&\leq\ex\brk{\frac{\pr\brk{\fS\mid\fC,\fD,\fI}}{\pr\brk{\fR\mid\fC,\fD,\fI}}\mid\fC,\fD}\leq q^{-m\bar\gamma(\alpha,\unfrozen)+o(n)},
	\end{align*}
	as desired.
\end{proof}

\begin{proof}[Proof of \Prop~\ref{prop_annealedwp}]
	As a first step we relate the number of $\alpha$-WP fixed points of $\vhA$ to the number of $\alpha$-covers.
	Given $\fD,\fS$ the random matrix $A(\vec\pi)$ has the same distribution as $\vhA$.
	Hence, suppose that $\fm$ is an $\alpha$-WP fixed point of $A(\vec\pi)$.
	Then $\fm$ induces a map $\fp_{\vec\pi}:\fV\cup\fF\to\{\frozen,\unfrozen\}^2$ by letting $\fp_{\vec\pi}(a_i,h)=(\fm_{v_j\to a_i},\fm_{a_i\to v_j})$, where $j\in[n]$ is the unique index such that $\vec\pi(a_i,h)\in\{v_j\}\times\NN$.
	Similarly, $\fp_{\vec\pi}(v_j,h)=(\fm_{a_i\to v_j},\fm_{v_j\to a_i})$ if $\vec\pi(v_j,h)\in\{a_i\}\times[k]$.
	The definitions \eqref{eqbiDelta1}--\eqref{eqbiDelta2} and \eqref{eqalphaWP1}--\eqref{eqalphaWP2} ensure that $(\vec\pi,\fp_{\vec\pi})$ satisfies {\bf COV1--COV4}.
	Thus, $(\vec\pi,\fp_{\vec\pi})$ is an $\alpha$-cover.
	
	Before we proceed we need to deal with an overcounting issue.
	Specifically, given $\fD$ for any matrix $\vhA$ there are $\Xi=(k!)^m\prod_{j=1}^nd_{\vhA}(v_j)!$ pairings $\pi$ that render $\vhA$, i.e., that satisfy $A(\pi)=\vhA$.
	At the same time, there are a total of $(km)!$ pairings $\pi$, and $\vhA$ and $A(\vec\pi)$ are identically distributed given $\fS$.
	In effect, \Lem~\ref{lemma_cover}, which counts the total number of $\alpha$-covers, implies that the number $W_\alpha$ of $\alpha$-WP fixed points of $\vhA$ satisfies
	\begin{align}\label{eqW}
		\ex[W_\alpha\mid\fD]=\exp(o(n))&&\mbox{\whp}
	\end{align}

	Now consider an extension $\sigma$ of $\fm$.
	Then $\sigma$ is {\em nearly} compatible with $(\vec\pi,\fp_{\vec\pi})$, except that \eqref{eqprop_annealedwp} may be violated for $o(n)$ indices $j\in[n]$.
	To remedy this set $\tau_j=\vecone\{\fp(v_j)=\unfrozen\}\sigma_j$.
	Then $\tau$ is compatible with $(\vec\pi,\fp)$ and \eqref{eqextension} implies
	\begin{align}\label{eqsigmatau1}
		\sum_{j=1}^n\vecone\{\sigma_j\neq\tau_j\}=o(n).
	\end{align}
	Further, because $\sigma\in\ker\vhA$, Fact~\ref{fact_degs} and \eqref{eqsigmatau1} yield $\|A(\vec\pi)\tau\|_0=o(n)$.
	Hence, $\tau$ essentially satisfies $(\vec\pi,\fp_{\vec\pi})$.
	
	Since \eqref{eqsigmatau1} shows that the number of inverse images $(\fm,\sigma)$ that can give rise to a specific pair $(\fp_{\vec\pi},\tau)$ is bounded by $\exp(o(n))$, in order to bound $\vX_\alpha$ it suffices to bound the expected number of pairs $(\fp_{\vec\pi},\tau)$ given $\fD$.
	The estimate \eqref{eqW} shows that the expected number of $\alpha$-covers $\fp_{\vec\pi}$ induced by $\alpha$-WP fixed points is bounded by $\exp(o(n))$.
	Furthermore, given $\fp_{\vec\pi}$ the number of assignments $\tau$ that satisfy the condition \eqref{eqprop_annealedwp0} is bounded by $q^{\bar\delta(\alpha,\unfrozen)n+o(n)}$.
	Moreover, \Lem~\ref{lem_extension} shows that such a $\tau$ is essentially satisfying with probability $q^{\bar\gamma(\alpha,\unfrozen)m+o(n)}$.
	Combining these estimates and recalling the definitions \eqref{eqPhi}, \eqref{eqdeltagamma1} and \eqref{eqdeltagamma2} of $\Phi$, $\bar\delta(\alpha,\unfrozen)$ and $\bar\gamma(\alpha,\unfrozen)$, we obtain
	\begin{align*}
		\ex\brk{\vX_\alpha\mid\fD}&\leq q^{\bar\delta(\alpha,\unfrozen)n+\bar\gamma(\alpha,\unfrozen)m+o(n)}=q^{\Phi_{d,k}(\alpha)n+o(n)}\qquad\mbox{\whp},
	\end{align*}
	thereby completing the proof.
\end{proof}

\subsection{Proof of \Cor~\ref{cor_annealedwp}}\label{sec_cor_annealedwp}
Fact~\ref{fact_phi} shows that for $d<d_k$ the function $\Phi_{d,k}(\alpha)$ attains its unique global maximum at $\alpha=0$.
Moreover, a glimpse at \eqref{eqPhi} reveals that $\Phi_{d,k}(0)=1-d/k$.
Hence, for any $d<d_k$ there exists $\zeta>0$ such that for any fixed $\xi>0$ we have $n\max_{\alpha\in[\xi,1]}\Phi(\alpha)<n-m-3\zeta n$.
Hence, \Prop s~\ref{prop_vhA} and~\ref{prop_annealedwp} show together with Markov's inequality that
\begin{align}\label{eqcor_annealedwp1}
	\pr\brk{\nul\vhA\geq n-m-2\zeta n\mid\vha\in[\xi,1]}&=\pr\brk{\max_{\alpha\in[\xi,1]}\vX_\alpha\geq q^{n-m-\zeta n}}+o(1)=o(1).
\end{align}
But since $\vhA$ has $m+o(n)$ rows, we have $\nul\vhA\geq n-m+o(n)$.
Therefore, \eqref{eqcor_annealedwp1} shows that $\vec\alpha<\xi$ \whp\
Letting $\xi\to0$ sufficiently slowly as $n\to\infty$, we thus conclude that $\vec\alpha=o(1)$ \whp\
Therefore, the assertion follows from \Prop~\ref{prop_vhA}.

\subsection{Proof of \Lem~\ref{lem_annealed_trivial}}\label{sec_lem_annealed_trivial}
Recall that for $\sigma \in \FF_q^n$ we let $\rho(\sigma) = (\rho_s(\sigma))_{s \in \FF_q}$ with $\rho_s(\sigma) = \frac 1 n \sum_{j = 1 }^{n} \vecone \{ \sigma_j = s \}$.
Let $\cR=\{\rho(\sigma):\sigma\in\FF_q^n\}$ be the set of all conceivable $\rho(\sigma)$-vectors.
Further, for $\chi = ( \chi_1, \dots , \chi_n ) \in \FF_q^n$ let $\chi^\perp=\{\sigma\in\FF_q^n:\sum_{j=1}^n\sigma_j\chi_j=0\}$. 
The following claim yields the approximate probability that a random vector whose entries are drawn independently from a distribution $r\in\cR$ close to the uniform distribution $q^{-1}\vecone$ belongs to $\chi^\perp$.

\begin{claim}\label{claim_lm}
	Let $\chi\in\FF_q^n$ be a vector with $|\supp\chi|=k\geq3$.
	Then uniformly for $r\in\cR$ with $\|r-q^{-1}\vecone\|<\eps$ we have
	\begin{align*}
		\varphi_\chi(r)=\sum_{\sigma\in\chi^\perp}\prod_{s\in\FF_q}r_s^{n\rho_s(\sigma)}=\frac1q+O(\eps^3)\qquad\mbox{as }\eps\to0.
	\end{align*}
\end{claim}
\begin{proof}
	Let $\cX(\chi)=\{\sigma\in\FF_q^{\supp\chi}:\sum_{j\in\supp\chi}\sigma_j\chi_j=0\}$ and for $\sigma\in\cX(\chi)$ and $s\in\FF_q$ let $R_s(\sigma)=|\{j\in\supp\chi:\sigma_j=s\}|$.
	Then $\varphi_\chi(r)=f_\chi(r)$, where
	$$f_\chi(r)=\sum_{\sigma\in\cX(\chi)}\prod_{s\in\FF_q}r_s^{R_s(\sigma)}.$$
	We are going to expand $f_\chi(r)$ to the second order.
	Clearly, $f_\chi(q^{-1}\vecone)=q^{-1}$, because $\cX(\chi)\subseteq\FF_q^{\supp\chi}$ is a linear subspace of codimension one and thus $|\cX(\chi)|=\FF_q^{k-1}$.
	Further, the partial derivatives of $f_\chi(r)$ come out as
	\begin{align}\label{eqclaim_lm1}
		\frac{\partial f_\chi }{ \partial r_t} &= \sum_{\sigma \in \cX(\chi)}  R_t(\sigma)r_t^{ R_t(\sigma) - 1}\prod_{ s   \in \FF_q\setminus\{t\} } r_s^{ R_s(\sigma) }     &&(t\in\FF_q),\\
		\frac{\partial^2 f_\chi }{ \partial r_t \partial r_u}   &= \sum_{\sigma \in \cX(\chi)} R_t(\sigma)R_u(\sigma)r_t^{R_t(\sigma)-1}r_u^{R_u(\sigma)-1} \prod_{s   \in \FF_q\setminus \{t,u \} } r_s^{ R_s(\sigma)} &&(t,u\in\FF_q,\,t\neq u),\label{eqclaim_lm2}\\
		\frac{\partial^2 f_\chi }{ \partial r_t^2}  &= \sum_{\sigma \in\cX(\chi)} R_t(\sigma)(R_t(\sigma)-1)r_t^{R_t(\sigma)-2} \prod_{ s   \in \FF_q\setminus \{t\} } r_s^{ R_s(\sigma) }&&(t\in\FF_q). \label{eqclaim_lm3}
	\end{align}
	To evaluate \eqref{eqclaim_lm1} at $r=q^{-1}\vecone$, we observe that the affine subspace $\{\sigma\in\cX(\chi):\sigma_j=t\}$ has dimension $k-2$ for every $t\in\FF_q$ and $j\in\supp\chi$, because $k=|\supp\chi|\geq3$.
	Hence,
	\begin{align}\label{eqclaim_lm4}
	\frac{\partial f_\chi }{ \partial r_t} \bigg\rvert_{r = q^{-1}\vecone  } &= q^{1-k} \sum_{j \in\supp\chi }\sum_{\sigma \in \cX(\chi) }    \vecone \{ \sigma_j = t \} =  \frac k q.
\end{align}
Similarly, since the affine subspaces $\{\sigma\in\cX(\chi):\sigma_j=t,\sigma_{j'}=u\}$ for $t,u\in\FF_q$ and $j,j'\in\supp\chi$, $j\neq j'$, have dimension $k-3$, \eqref{eqclaim_lm2}--\eqref{eqclaim_lm3} evaluated at $r=q^{-1}\vecone$ boil down to
\begin{align}\label{eqclaim_lm5}
\frac{\partial^2 f_\chi }{ \partial r_t \partial r_u}\bigg\rvert_{r = q^{-1}\vecone  }&=\frac{\partial^2 f_\chi }{ \partial^2 r_t}\bigg\rvert_{r = q^{-1}\vecone  }=\frac{k(k-1)}q.
\end{align}
Further, all third partial derivatives remain bounded, i.e.,
\begin{align}\label{eqclaim_lm6}
	\frac{\partial^3 f_\chi }{ \partial r_s \partial r_t \partial r_u}&=O(1)&&\mbox{for all }s,t,u\in\FF_q.
\end{align}
Finally, since for every $r\in\cR$ we have $\sum_{s\in\FF_q}r_s=1$ and the only eigenspaces with non-zero eigenvalues of the Jacobi matrix $Df_\chi(q^{-1}\vecone)$ and of the Hessian $D^2_\chi(q^{-1}\vecone)$ are spanned by $\vecone$, the assertion follows from \eqref{eqclaim_lm4}--\eqref{eqclaim_lm6} and Taylor's formula.
\end{proof}

\begin{proof}[Proof of \Lem~\ref{lem_annealed_trivial}]
	Given the value of $\vt$ the random matrix $\vhA$ consists of $m$ rows of support size $k$ and $\vt$ unary rows.
	These rows are stochastically independent.
	Therefore, Claim~\ref{claim_lm} shows that for any $r\in\cR$ and any $\sigma\in\FF_q^n$ with $\rho(\sigma)=r$ we have
	\begin{align}\label{eqlem_annealed_trivial1}
		\pr\brk{\sigma\in\ker\vhA\mid\vt}&=q^{-m-\vt}\exp(O(n\|r-q^{-1}\vecone\|_1^3)).
	\end{align}
	Further, we recall that the entropy function $H(r)$ has the expansion
	\begin{align}\label{eqlem_annealed_trivial2}
		H(r)=\log q-\frac q2\sum_{s\in\FF_q}(r_s-q^{-1})^2+O(\|r-q^{-1}\vecone\|_1^3).
	\end{align}
	Combining \eqref{eqlem_annealed_trivial1}--\eqref{eqlem_annealed_trivial2} and applying the Laplace method, we thus obtain for small enough $\eps>0$,
\begin{align*}
	\ex&\abs{\ker\vhA\cap\{\sigma\in\FF_q^n:\|\rho(\sigma)-q^{-1}\vecone\|_1<\eps\}\mid\vt}\\&= (1+o(1))q^{n-m-\vt}\sum_{r\in\cR:\|r-q^{-1}\vecone\|_1<\eps}
	\frac{\exp(-q\|r-q^{-1}\vecone\|_2^2/2+O(\|r-q^{-1}\vecone\|_1^3))}{\sqrt{(2\pi n)^{q-1}\prod_{s\in\FF_q}r_s}}\sim q^{n-m-\vt},
	\end{align*}
	as claimed.
\end{proof}

\section{Proof of \Thm~\ref{thm_main} (ii)}\label{sec_finish}

\noindent
The proof of the second part of \Thm~\ref{thm_main} is based on the interpolation method from mathematical physics~\cite{PanchenkoTalagrand}.
The interpolation method has been applied previously in order to estimate the rank of random matrices from a more general model~\cite{Maurice}, and in fact the upper bound on the rank obtained in~\cite{Maurice} implies \Thm~\ref{thm_main} (ii).
Nonetheless, for the sake of completeness here we present a simplified version of the interpolation argument tailored to the specific random matrix model $\vhA$.

The basic idea is to construct a family $\vhA(\theta)$ of matrices parametrised by $\theta\in[0,1]$.
The first matrix $\vhA(0)$ (essentially) coincides with the random matrix $\vhA$, while at the other end $\vhA(1)$ we have a matrix whose nullity is easy to compute explicitly.
We will then differentiate $\ex[\nul\vhA(\theta)]$ to compare $\ex[\nul\vhA(0)]$ and $\ex[\nul\vhA(1)]$.
Thus, we obtain a lower bound on the nullity of $\vhA(0)$, and hence of $\vhA$.
Since $\nul(\vhA)+\rk(\vhA)=n$, this lower bound on the nullity translates into the desired upper bound on the rank of $\vhA$.

The interpolating family $\vhA(\theta)$ is constructed as follows.
Let $\vm_\theta,\vm_\theta'$ be two independent Poisson variables with means $(1-\theta)dn/k$ and $d\theta\alpha_\frozen^{k-1}n$, respectively; here $\alpha_\frozen=\alpha_\frozen(d,k)>0$ is the maximum fixed point of $\phi_{d,k}$ (see Fact~\ref{fact_phi}).
Both $\vm_\theta,\vm_\theta'$ are also independent of the uniform random variable $\vt\in[T]$.
The random matrix $\vhA(\theta)$ has size $(\vm_\theta+\vm_\theta'+\vt)\times n$.
As in the definition \eqref{eqA} of $\vA$, the first $\vm_\theta$ rows of $\vhA(\theta)$ have entries
\begin{align*}
	\vhA_{ij}(\theta)&=\fA_{ij}\vecone\{j\in\ve_i\}&&(i\in[\vm_\theta],j\in[n]),
\end{align*}
where $(\ve_i)_{i\geq1}$ is a family of uniformly random subsets of $[n]$ of size $k$; these sets are mutually independent as well as independent of $\vm_t,\vm_t'$ and $\vt$.
Further, for $\vm_t<i\leq1+\vm_t'+\vt$ the $i$-th row of $\vhA$ contains a single one in a uniformly random column $j\in[n]$, while all other entries are zero.
The positions of these $1$-entries are drawn independently of each other and of everything else.

\begin{lemma}\label{lem_interpol_extreme}
	We have $\ex[\nul\vhA(0)]=\ex[\nul\vA]+o(n)$ and $\ex[\nul\vhA(1)]=n\exp(-d\alpha_\frozen^{k-1})+o(n)$.
\end{lemma}
\begin{proof}
	By construction the first $\vm_0\wedge m$ rows of $\vhA(0)$ and $\vA$ are identically distributed.
	Moreover, \whp\ we have $\vm_0=m+o(n)$.
	Since adding or removing a single row can alter the nullity by at most one, the first assertion follows.

	Regarding the second assertion, observe that the rows of $\vhA(1)$ are all-zero, except for a single one entry that sits in an independent and uniformly random position.
	Hence, the nullity of $\vhA(1)$ is simply the number of all-zero columns.
	Further, since $\ex[\vm_1']=d\alpha_\frozen^{k-1}n$, the expected number of non-zero entries per column equals $d\alpha_\frozen^{k-1}+o(1)$.
	Since the $\vm_\theta'$ is a Poisson variable, we expect $n\exp(-d\alpha_\frozen^{k-1}+o(1))$ all-zero columns.
\end{proof}

The main step of the interpolation method is to compute the derivative $\frac\partial{\partial \theta}\ex[\nul\vA(\theta)]$.

\begin{lemma}\label{lem_interpol_deriv}
	We have $\frac1n\frac\partial{\partial \theta}\ex[\nul\vA(\theta)]\leq -d\alpha_\frozen^{k-1}+\frac dk(k-1)\alpha_\frozen^{k}+\frac dk+o(1)$.
\end{lemma}
\begin{proof}
	Since $\vm_\theta,\vm_\theta'$ are Poisson variables, we calculate
	\begin{align*}
		\frac1n\frac\partial{\partial \theta}\pr\brk{\vm_\theta=m}&=\frac{d}k\brk{\pr\brk{\vm_\theta=m}-\pr\brk{\vm_\theta=m-1}},&
		\frac1n\frac\partial{\partial \theta}\pr\brk{\vm_\theta'=m}&=d\alpha_\frozen^{k-1}\brk{\pr\brk{\vm_\theta'=m-1}-\pr\brk{\vm_\theta'=m}}.
	\end{align*}
	Therefore, 
	\begin{align}\nonumber
		\frac1n\frac\partial{\partial \theta}\ex[\nul\vA(\theta)]&=\frac1n\sum_{m,m'\geq0}\ex\brk{\nul\vhA(\theta)\mid\vm_\theta=m,\vm_\theta'=m'}\frac\partial{\partial \theta}\pr\brk{\vm_\theta=m}\pr\brk{\vm_\theta'=m'}\\
																 &=d\alpha_\frozen^{k-1}\sum_{m'\geq0}\brk{\ex\brk{\nul\vhA(\theta)\mid\vm_\theta'=m'+1}-\ex\brk{\nul\vhA(\theta)\mid\vm_\theta'=m'}}\pr\brk{\vm_\theta'=m'}\nonumber\\
																 &\quad-\frac dk\sum_{m\geq0}\brk{\ex\brk{\nul\vhA(\theta)\mid\vm_\theta=m+1}-\ex\brk{\nul\vhA(\theta)\mid\vm_\theta=m}}\pr\brk{\vm_\theta=m}.
																 \label{eq_lem_interpol_deriv1}
	\end{align}
	Hence, obtain $\vhA_+(\theta)$ from $\vhA(\theta)$ by adding one more row with precisely one non-zero entry in a uniformly random position, chosen independently of everything else.
	Let $\va^+$ signify this new row.
	Similarly, obtain $\vhA_-(\theta)$ from $\vhA(\theta)$ by adding the row $\va^-$ with entries 
		$$\va^-_j=\fA_{\vm_\theta+1\,j}\vecone\{j\in\ve_{m_\theta+1}\}.$$
	Then \eqref{eq_lem_interpol_deriv1} shows that
	\begin{align} \label{eq_lem_interpol_deriv2}
		\frac1{dkn}\frac\partial{\partial \theta}\ex[\nul\vA(\theta)]&=k\alpha_\frozen^{k-1}\ex\brk{\nul(\vhA_+(\theta))-\nul(\vhA(\theta))}-\ex\brk{\nul(\vhA_-(\theta))-\nul(\vhA(\theta))}.
	\end{align}

	Let $\vha_\theta=|\cF(\vhA(\theta))|/n$.
	We claim that
	\begin{align}\label{eq_lem_interpol_deriv3}
		\ex\brk{\nul(\vhA_+(\theta))-\nul(\vhA(\theta))}&=-\ex\brk{1-\vha_\theta}.
	\end{align}
	Indeed, let $\vj^+\in[n]$ be the position of the non-zero entry of $\va^+$.
	Then adding $\va^+$ to $\vhA(\theta)$ decreases the nullity iff $j^+\not\in\cF(\vhA(\theta))$.
	Since $\vj^+$ is uniformly random and independent of $\vhA(\theta)$, we obtain \eqref{eq_lem_interpol_deriv3}.

	Further, we claim
	\begin{align}\label{eq_lem_interpol_deriv4}
		\ex\brk{\nul(\vhA_-(\theta))-\nul(\vhA(\theta))}&=-\ex\brk{1-\vha_\theta^k}+o(1).
	\end{align}
	To see this, let $\cE_\theta$ be the event that $\vhA(\theta)$ is $(o(1),k)$-free.
	Since the construction of $\vhA(\theta)$ incorporates $\vt$ random unary equations as in the pinning lemma (\Lem~\ref{lem_pinning}), we have $\pr\brk{\cE_\theta}=1-o(1)$.
	Furthermore, since $\va^-$ is independent of $\vhA(\theta)$, the probability that the positions $1\leq\vj_1^-<\cdots<\vj_k^-\leq n$ of the non-zero entries of $\va^-$ form a proper relation of $\vhA(\theta)$ is $o(1)$ on the event $\cE_\theta$.
	Hence, assume that $\vj_1^-,\ldots,\vj_k^-$ do not form a proper relation.
	Then the nullity drops upon addition of row $\va^-$ unless $\vj_1^-,\ldots,\vj_k^-\in\cF(\vhA(\theta))$.
	Since $\pr\brk{\vj_1^-,\ldots,\vj_k^-\in\cF(\vhA(\theta))\mid\vhA(\theta)}=\vha_\theta^k+o(1)$, we obtain \eqref{eq_lem_interpol_deriv4}.

	Combining \eqref{eq_lem_interpol_deriv2}--\eqref{eq_lem_interpol_deriv4}, we find
\begin{align} \label{eq_lem_interpol_deriv5}
	\frac1{dkn}\frac\partial{\partial \theta}\ex[\nul\vA(\theta)]&=\ex\brk{1-\vha_\theta^k-k\alpha_\frozen^{k-1}(1-\vha_\theta)}+o(1).
	\end{align}
	To complete the proof, we notice that
	\begin{align}\label{eq_lem_interpol_deriv6}
		1-\vha_\theta^k-k\alpha_\frozen^{k-1}(1-\vha_\theta)+\bc{k\alpha_\frozen^{k-1}-(k-1)\alpha_\frozen^{k}-1}=-\vha_\theta^k+k\vha_\theta\alpha_\frozen^{k-1}-(k-1)\alpha_\frozen^{k}\leq0,
	\end{align}
	because $X^k-kXY^{k-1}+(k-1)Y^k\geq0$ for all $X,Y\in[0,1]$ and all $k\geq2$.
	The assertion follows from \eqref{eq_lem_interpol_deriv5} and \eqref{eq_lem_interpol_deriv6}.
\end{proof}

\begin{proof}[Proof of \Thm~\ref{thm_main} (ii)]
	Suppose that $d>d_k$.
	Integrating on $\theta\in[0,1]$, we learn from Fact~\ref{fact_phi} and \Lem~\ref{lem_interpol_extreme} that
	\begin{align}\label{eqthm_main1}
		\frac1n\ex[\nul\vhA]&\geq\Phi_{d,k}(\alpha_\frozen)+o(1)>1-d/k.
	\end{align}
	Furthermore, Azuma--Hoeffding shows that $\nul\vhA$ is tightly concentrated, because adding or removing a single row alters the nullity by at most one.
	Thus, since $\vhA$ is obtained from $\vA$ via the addition of $o(n)$ rows, we conclude that $n^{-1}\nul\vA\geq\Phi_{d,k}(\alpha_\frozen)+o(1)$ \whp\
	Therefore, \eqref{eqthm_main1} shows that $\rank\vA<m-\Omega(n)$ \whp
\end{proof}


\begin{thebibliography}{99}
	\bibitem{AchlioptasMolloy} D.\ Achlioptas, M.\ Molloy: The solution space geometry of random linear equations.  Random Structures and Algorithms {\bf46} (2015) 197--231.
	\bibitem{ANP} D.\ Achlioptas, A.\ Naor, Y.\ Peres: Rigorous location of phase transitions in hard optimization problems. Nature {\bf 435} 759--764.
	\bibitem{Aizenman} M.\ Aizenman, R.\ Sims, S.\ Starr: An extended variational principle for the SK spin-glass model.  Phys.\ Rev.\ B {\bf68}  (2003) 214403.
	\bibitem{Ayre} P.\ Ayre, A.\ Coja-Oghlan, P.\ Gao, N.\ M\"uller: The satisfiability threshold for random linear equations.  Combinatorica {\bf40} (2020) 179--235.
	\bibitem{parity} A.\ Coja-Oghlan, O.\ Cooley, M.\ Kang, J.\ Lee, J.\ Ravelomanana: The sparse parity matrix. Proc.\ 33rd SODA (2022) 822--833.
	\bibitem{COGHKLMR} A.\ Coja-Oghlan, P.\ Gao, M.\ Hahn-Klimroth, J.\ Lee, N.\ M\"uller, M.\ Rolvien: The full rank condition for sparse random matrices. arxiv 2112.14090 (2021).
	\bibitem{Maurice} A.~Coja-Oghlan, A.~Erg\"ur, P.~Gao, S.~Hetterich, M.~Rolvien: The rank of sparse random matrices. Proc.\ 31st SODA (2020) 579--591.
	\bibitem{OJJ} O.~Cooley,  J.\ Lee, J.\ Ravelomanana: Warning Propagation: stability and subcriticality.  arXiv:2111.15577 (2021).
	\bibitem{Cooper} C.\ Cooper: The cores of random hypergraphs with a given degree sequence.  Random Structures and Algorithms {\bf25} (2004) 353--375.
	\bibitem{CFP} C.\ Cooper, A.\ Frieze, W.\ Pegden: On the rank of a random binary matrix.  Electron.\ J.\ Comb.\ {\bf26} (2019) P4.12.
	\bibitem{CDD} N.\ Creignou, H.\ Daude, O.\ Dubois: Approximating the satisfiability threshold for random $k$-XOR-formulas. arXiv:cs/0106001 (2001).
	\bibitem{Dietzfelbinger} M.\ Dietzfelbinger, A.\ Goerdt, M.\ Mitzenmacher, A.\ Montanari, R.\ Pagh, M.\ Rink: Tight thresholds for cuckoo hashing via XORSAT.  Proc.\ 37th ICALP (2010) 213--225.
	\bibitem{DSS3} J.~Ding, A.~Sly, N.~Sun: Proof of the satisfiability conjecture for large $k$.  Annals of Mathematics {\bf196} (2022) 1--388.
	\bibitem{DuboisMandler} O.\ Dubois, J.\ Mandler: The 3-XORSAT threshold. Proc.\  43rd FOCS (2002) 769--778.
	\bibitem{Fernholz2} D.\ Fernholz, V.\ Ramachandran: Cores and connectivity in sparse random graphs.  UTCS Technical Report TR04-13 (2004).
	\bibitem{GoerdtFalke} A.\ Goerdt, L.\ Falke: Satisfiability thresholds beyond $k$-XORSAT.  Proc.\ 7th International Computer Science Symposium in Russia (2012) 148--159.
	\bibitem{Ibrahimi} M.\ Ibrahimi, Y.\ Kanoria, M.\ Kraning, A.\ Montanari: The set of solutions of random XORSAT formulae. Annals of Applied Probability {\bf 25} (2015) 2743--2808.
	\bibitem{JansonLuczak} S.\ Janson, M.\ Luczak: A simple solution to the $k$-core problem.  Random Structures and Algorithms  {\bf 30} (2007) 50--62.
	\bibitem{JLR} S.~Janson, T.~Luczak, A.~Rucinski: Random graphs. Wiley (2000).
\bibitem{Kim} J.H.~Kim: Poisson cloning model for random graphs.  Proceedings of the International Congress of Mathematicians (2006) 873--897.
	\bibitem{MM} M.~M\'ezard, A.~Montanari: Information, physics and computation.  Oxford University Press (2009).
	\bibitem{MRTZ} M.\ M\'ezard, F.\ Ricci-Tersenghi, R.\ Zecchina: Two solutions to diluted $p$-spin models and XORSAT problems.  Journal of Statistical Physics {\bf 111} (2003) 505--533.
\bibitem{Mike} M.\ Molloy: Cores in random hypergraphs and Boolean formulas.  Random Structures and Algorithms {\bf 27} (2005) 124--135.
	\bibitem{Montanari} A.\ Montanari: Estimating random variables from random sparse observations. European Transactions on Telecommunications {\bf19}(4) (2008) 385--403.
\bibitem{PanchenkoTalagrand} D.\ Panchenko, M.\ Talagrand: Bounds for diluted mean-fields spin glass models.  Probab.\ Theory Relat.\ Fields {\bf130} (2004) 319--336.
	\bibitem{PittelSorkin} B.\ Pittel, G.\ Sorkin:  The satisfiability threshold for $k$-XORSAT.  \CPC\ {\bf25} (2016) 236--268.
	\bibitem{Pittel} B.\ Pittel, J.\ Spencer, N.\ Wormald: Sudden emergence of a giant $k$-core in a random graph.  Journal of Combinatorial Theory, Series B {\bf 67} (1996) 111--151.
	\bibitem{Raghavendra} P.\ Raghavendra, N.\ Tan: Approximating CSPs with global cardinality constraints using SDP hierarchies. Proc.\ 23rd SODA (2012) 373--387.
	\bibitem{Riordan} O.\ Riordan: The $k$-core and branching processes. Combinatorics, Probability and Computing {\bf 17} (2008) 111--136.
\end{thebibliography}
\end{document}